\documentclass[reqno]{amsart} 
\usepackage{amsmath}

\usepackage{verbatim}
\usepackage{amsthm}
\usepackage{amssymb}
\usepackage{enumerate}

    \usepackage{geometry}
\newtheorem{theorem}{Theorem}[section]
\newtheorem{example}[theorem]{Example}

\newtheorem{lemma}[theorem]{Lemma}
\newtheorem{corollary}[theorem]{Corollary}
\newtheorem{question}[theorem]{Question}
\newtheorem{claim}{Claim}[theorem]
\newtheorem{proposition}[theorem]{Proposition}
\newtheorem{conjecture}[theorem]{Conjecture}
\newtheorem{obs}[theorem]{Observation}
\newtheorem{fact}[theorem]{Fact}

\newtheorem{Claim}{Claim}

\newcommand{\newClaim}{\setcounter{Claim}{0}}

\theoremstyle{definition}
 
\newtheorem{definition}[theorem]{Definition}

\newcommand{\mc}[1]{\mathcal{#1}}
\newcommand{\mbb}[1]{\mathbb{#1}}

\newcommand{\mf}[1]{\mathfrak{#1}}

\newcommand{\setm}{\setminus}
\newcommand{\empt}{\emptyset}
\newcommand{\subs}{\subset}

\newcommand{\dom}{\operatorname{dom}}

\def\<{\left\langle}
\def\>{\right\rangle}
\def\br#1;#2;{\bigl[ {#1} \bigr]^ {#2} }
\newcommand{\oo}{{{\omega}_1}}
\newcommand{\supp}[1]{\textrm{supp}}

\newcommand{\ssigma}{\sigma^*}

\usepackage{tikz}

\newcommand{\snode}[2]{\node (#1) [ pro] {\tiny #2};}

\title{Comparing weak versions of separability}

\author{D\'aniel Soukup}
\address{Department of Mathematics, University of Toronto, Toronto, ON}
\email{daniel.soukup@mail.utoronto.ca}
\author{Lajos Soukup}
\address{R\'enyi Institute of Mathematics, Hungarian Academy of Sciences, Budapest, Hungary}
\email{soukup@renyi.hu}
\urladdr{http://www.renyi.hu/$\sim$soukup}
\author{Santi Spadaro}
\address{Department of Mathematics and Statistics, Faculty of Science and Engineering, York University, Toronto, ON}
\email{sspadaro@mathstat.yorku.ca, santispadaro@yahoo.com}
\thanks{The second author was supported by  the Hungarian National Research Grant OTKA  K 83726.
The third named author was partially supported by an European Science Foundation INFTY grant, by the Center for Advanced Studies at Ben Gurion University of the Negev and by an INdAM Cofund outgoing fellowship}


\begin{document}

\begin{abstract}
 Our aim is to investigate spaces with $\sigma$-discrete and meager dense sets, as well as selective versions of these properties. We construct numerous examples to point out the differences between these classes while answering questions of Tkachuk \cite{Tkdsep}, Hutchinson \cite{Hutch} and the authors of \cite{BMS}.
\end{abstract}
\maketitle

\section{Introduction}

Topologists and analysts have often considered properties stating that a space has a \emph{small} dense set. The most popular of them is \emph{separability}, that is, the property of having a countable dense set. The famous \emph{Suslin Problem} asked whether there is a non-separable linearly ordered space where families of pairwise disjoint open sets are at most countable and the still open \emph{Separable Quotient Problem} asks whether every infinite dimensional Banach space has an infinite dimensional separable quotient.

Smallness conditions for dense sets other than separability have also been
considered. A space is called \emph{d-separable} if it has a dense set which is
the countable union of discrete subsets. This property was introduced by Kurepa
in his PhD dissertation under the name of \emph{property $K_0$} as part of his
study of the Suslin Problem. The latter can in fact be restated to ask whether
there is a non-d-separable linearly ordered space where families of open sets
are at most countable. $d$-separability has a much better behavior than
separability: arbitrary products of $d$-separable spaces are $d$-separable, and
for every space $X$ there is a cardinal $\kappa$ such that $X^\kappa$ is
$d$-separable. First, we will introduce a natural property called
\emph{$nwd$-separability} which is obtained by replacing \emph{discrete} with
\emph{nowhere dense} in the definition. Every $d$-separable space without
isolated points is $nwd$-separable, and $nwd$-separability shows a behavior
which is somewhat close to that 
of $d$-separability. 

A new class of smallness conditions for dense sets has been introduced as part of the program known as \emph{selection principles in mathematics} and attracted a lot of attention recently, see \cite{BBM}, \cite{BMS}, \cite{GS} or \cite{Sak} among others. The general idea is that a small dense set can be obtained by diagonalizing over a countable sequence of dense sets. In this way, one can define a selective strengthening of any of the properties we mentioned above: a space is \emph{$D$-separable} \cite{BMS}  iff for every sequence $\{D_n: n < \omega\}$ of dense sets there is a discrete set $E_n \subset D_n$ for every $n< \omega$ such that $\bigcup_{n<\omega} E_n$ is dense. We define $NWD$-separability as a selective version of $nwd$-separability in a similar way and compare it with $D$-separability.

We devoted most of our efforts to point out various differences between these properties and to construct a great wealth of examples.

In Section \ref{nonsel}, we introduce $nwd$-separability and start with pointing out some facts concerning products. We continue by proving that $\omega^*\times 2^\omega$ is a compact space which is $nwd$-separable but not $d$-separable.
 The section ends with answering a question of Tkachuk \cite{Tkdsep} by showing that there is 
a Corson-compact space with non-d-separable square.

Next, in Section \ref{sel} we begin to deal with selective versions of separability. We present a new construction of a countable $M$-separable, non-$R$-separable space which also serves as an answer to a question of Hutchinson \cite{Hutch}. We present a general framework to deal with selective separability properties and conclude that the class of $D$-and $NWD$-separable spaces are close under finite unions.

In Section \ref{zfc}, our aim is to construct ZFC examples separating the newly introduced properties. We present an $NWD$-separable space which is not $d$-separable and countable, dense subsets of $2^\mathfrak{c}$ which are not $NWD$-separable. We finish by investigating some related cardinal invariants and answering several questions from \cite{BMS}.

Section \ref{force} is devoted to show, by forcing, that even in the class of first-countable spaces, $d$-and $D$-separability ($nwd$-and $NWD$-separability) can be different; compare this with the result that every separable Fr\'echet space is $M$-separable \cite{BD}.

Finally, in Section \ref{mn} we finish with some positive results: we prove that every monotone normal, $nwd$-separable space is $D$-separable and show that the additional assumption of compactness even yields a $\sigma$-disjoint $\pi$-base. The last part of the section deals with the question whether $\sigma(2^\oo)$ is $D$-separable.

\section{Non-selective properties} \label{nonsel}

We start by defining two natural weakening of separability. The first one has been studied extensively in the past.

\begin{definition}
A space is $d$-separable (respectively, $nwd$-separable) if there are discrete (respectively, nowhere dense) sets $\{D_n: n < \omega \}$ such that $\bigcup_{n<\omega} D_n$ is dense.
\end{definition}

Arhangel'skii \cite{Ardsep} proved that arbitrary products of $d$-separable
spaces are $d$-separable. Juh\'asz and Szentmikl\'ossy \cite{JS} proved that for
every space $X$, the space $X^{d(X)}$ is $d$-separable. Moreover, they proved
that for every compact space $X$, the countable power $X^\omega$ is
$d$-separable. However, the behavior of product spaces considering
$nwd$-separability is much simpler:

\begin{proposition} If $X$ is $nwd$-separable and $Y$ is arbitrary then $X\times Y$ is $nwd$-separable; thus finite products of $nwd$-separable spaces remain $nwd$-separable.

 $\prod \{X_\alpha:\alpha<\lambda\}$ is $nwd$-separable for arbitrary spaces $X_\alpha$  with $|X_\alpha|\ge 2$ and infinite $\lambda$.
\end{proposition}
\begin{proof}
Note that if $E\subseteq X$ is nowhere dense then $E\times Y$ is nowhere dense in $X\times Y$. Thus the first part clearly follows. 

Now, observe that $X=\prod \{X_n:n\in\omega\}$ is $nwd$-separable for arbitrary
spaces $X_n$ with $|X_n|\ge 2$; indeed, fix some $x_i\in X_i$ for $i\in\omega$ and define 
$D_n=\{y \in X: (\forall i \geq n) (y(i)=x_i) \}$. Note that $D_n$ is nowhere dense for
each $n\in\omega$ and $\bigcup_{n<\omega} D_n$ is dense in $X$. Now consider an
arbitrary infinite product $X=\prod \{X_\alpha:\alpha<\lambda\}$ and note that
$X$ is homeomorphic to a countably infinite product.
\end{proof}

Note that finite powers can be non $nwd$-separable. A space is called an \emph{almost $P$-space} if and only if every non-empty countable intersection of open sets has non-empty interior.

\begin{obs} \label{nond}
Let $X$ be a regular countably compact almost $P$-space. Then every meager set is nowhere dense in $X$; thus $X$ is not $nwd$-separable.
\end{obs}
\begin{proof}
 Fix nowhere dense sets $E_n\subseteq X$ for $n\in\omega$ and a nonempty open $V\subseteq X$. Construct a decreasing sequence of open sets $U_n\subseteq V\setminus E_n$ such that $\overline{U_{n+1}}\subseteq U_n$. Then $\bigcap_{n\in\omega} U_n=\bigcap_{n\in\omega} \overline{U_n}\neq \emptyset$ by $X$ being countably compact and thus there is a nonempty open $U\subseteq\bigcap_{n\in\omega} U_n$ by $X$ being an almost $P$-space. Thus $U\cap \bigcup_{n\in\omega} E_n=\emptyset$ which show that $\bigcup_{n\in\omega} E_n$ is not dense in $V$.
\end{proof}

Thus $(\omega^*)^n$ is not $nwd$-separable for $n\in\omega$ since every finite power of $\omega^*$ is a compact, almost $P$-space. J. Moore \cite{Md} showed that there is an $L$-space, i.e. hereditarily Lindel\"of, non separable space,
 with a $d$-separable square. Thus, there are non $d$-separable spaces with $d$-separable square.
Todorcevic, \cite{Tod2}, got the idea  that the  Ellentuck topology can show   the same situation for $nwd$-separability,
and his conjecture was correct:

\begin{example}The Ellentuck topology $\mathcal{X}=[\omega]^{\omega}$ is a first-countable, non $nwd$-separable space with $nwd$-separable square.
\end{example}
\begin{proof}
 Recall that the standard basis for the Ellentuck topology is $$\{[s,X]:s\in [\omega]^{<\omega}, X\in [\omega]^{\omega} \}$$ where  $[s,X]=\{ Y\in [\omega]^{\omega}: Y \text{ is an end-extension of } s, Y\setminus s\subseteq X \}$. It is well known, though not trivial, that every meager set in the Ellentuck topology is nowhere dense; thus $\mathcal{X}$ is not $nwd$-separable.

Let us construct a $D_{s,t}\subseteq \mathcal{X}^2$ for
$s,t\in[\omega]^{<\omega}$ as follows: for $A_0,A_1\in [\omega]^{\omega}$ we say
that $A_0$ and $A_1$ are \emph{merged} iff for all $i<2$ and $n,m\in A_i$ with
$n<m$ there is $k\in A_{1-i}$ such that $n<k<m$. Let 
$$D_{s,t}=\big\{(A,B)\in[s,\omega]\times[t,\omega]:A\setminus s \text{ and }
B\setminus t \text{ are merged}\big\}.$$ 
It can be easily seen that $D_{s,t}$ is
nowhere dense and that the meager set
$\bigcup\{D_{s,t}:s,t\in[\omega]^{<\omega}\}$ is dense in $\mathcal{X}^2$.
\end{proof}

Our next aim is to show that $d$-separability and $nwd$-separability are different properties even in the realm of compact spaces.

\begin{example}\label{ex:nwdnotd}
The space $X=\omega^* \times 2^\omega$ is a compact $nwd$-separable space which is not $d$-separable.
\end{example}

\begin{proof}
Indeed, let $D=\{x_n: n \in \omega \}$ be a countable dense subset of $2^\omega$. Then $\bigcup_{n<\omega} \omega^* \times \{x_n\}$ is a $\sigma$-nowhere dense and dense subset of $X$.

To see that $X$ is not $d$-separable, suppose by contradiction that 
$\{D_n: n < \omega \}$ is a countable family of discrete sets whose union is dense in $X$
and let $\mathcal{B}=\{B_m: m < \omega \}$ be a countable base for $2^\omega$,
$\tau$ be the topology of $\omega^*$ and $\pi: \omega^* \times 2^\omega \to
\omega^*$ be the projection onto the first coordinate. Let 
$$D_{nm}=\{\pi(z): z \in D_n \wedge  (\exists U \in \tau)((U \times B_m) \cap D_n=\{z\}) \}.$$

Then $D_{nm}$ is a discrete subset of $\omega^*$ and $\bigcup \{D_{nm}: (n,m) \in \omega \times \omega \}$ is dense in $\omega^*$. But this contradicts Observation $\ref{nond}$.
\end{proof}


Recall that a space $X$ is called a \emph{Corson compactum} if it is compact and there is a cardinal $\kappa$ such that $X$ can be embedded in $\Sigma(\mathbb{R}^\kappa)=\{x \in \mathbb{R}^\kappa: |\{\alpha<\kappa:x(\alpha)\neq 0\}| \leq \omega \}$. In \cite{Tkdsep} Tkachuk asked if the square of every Corson compactum is $d$-separable. We are going to show that Todorcevic's classical example of a Corson compactum is a counterexample to Tkachuk's question.

Given a tree $(T, \leq)$, we let $T \otimes T=\{(s,t): s, t \in T \wedge ht(s)=ht(t)\}$ and order $T \otimes T$ as follows $(s, t) \leq (s', t')$ if and only if $s \leq s'$ and $t \leq t'$.

Now fix a stationary co-stationary subset $A$ of $\omega_1$ and let $T$ be the tree of all countable compact subsets of $A$ ordered by $s \leq t$ if and only if $s$ is an initial part of $t$.

\begin{theorem}
(Todorcevic, \cite{Tod}) $T \otimes T$ is Baire in the final parts topology.
\end{theorem}

Let $P(T)$ be the set of all paths in $T$ with the topology inherited from $2^T$. It is easy to see that $P(T)$ is closed and thus $P(T)$ is compact. From the fact that $A$ is co-stationary it follows that every path is countable in $T$ and hence $P(T)\subseteq \Sigma(2^T)$. Thus $P(T)$ is a Corson compactum.

\begin{theorem}
The square of Todorcevic's Corson Compactum $P(T)$ is not $nwd$-separable (hence not $d$-separable either).
\end{theorem}
\begin{proof}
 Given $t \in T$, we let $U_t$ be the set of all paths passing through $t$. We note that $\{U_s \times U_t: (s,t) \in T \otimes T \}$ is a $\pi$-base for $P(T) \times P(T)$.

Now suppose that $P(T)^2$ has a dense set of the form $\bigcup_{n<\omega} D_n$, where each $D_n$ is nowhere dense.

\noindent {\bf Claim.} $W_n=\{(s,t) \in T \otimes T: (U_s \times U_t) \cap D_n=\emptyset\}$ is open dense in the final parts topology on $T \otimes T$.

\begin{proof}[Proof of Claim]To prove that $W_n$ is open, note that if $(U_s \times U_t) \cap D_n=\emptyset$ and $(s', t') \geq (s,t)$ then we also have $(U_{s'} \times U_{t'}) \cap D_n=\emptyset$.

To prove that $W_n$ is dense, let $(s,t) \in T \otimes T$ be arbitrary and note that $(U_s \times U_t )\setminus \overline{D_n}$ is a non-empty open  set, so we can find $s'$ and $t'$ such that
 $U_{s'} \times U_{t'} \subset (U_s \times U_t) \setminus \overline{D_n}$. But then $W_n \ni (s', t') \geq (s,t)$.
\end{proof}

By the Baire property of $T \otimes T$ we can choose $(s,t) \in \bigcap_{n<\omega} W_n$, but then we see that $(U_s \times U_t) \cap (\bigcup_{n<\omega} D_n)=\emptyset$, which is a contradiction.
\end{proof}

\section{Selective properties} \label{sel}

The following properties are the first selection principles for dense sets to have been considered and were introduced in \cite{S} under a different name.

\begin{definition}
A space is called {\em $M$-separable} ({\em $R$-separable}), if given a sequence $\{D_n: n < \omega \}$ of dense sets there are finite (one-point) sets $F_n \subset D_n$ such that $\bigcup_{n<\omega} F_n$ is dense in $X$.
\end{definition}

The standard way of constructing a $M$-separable non-$R$-separable space uses function spaces via the following theorem.

\begin{theorem}[{\cite[Theorems 21 and 57]{BBM}}]
Let $X$ be a Tychonoff Space. Then $C_p(X)$ is $M$-separable ($R$-separable) if and only if $C_p(X)$ is separable and $X^n$ is Rothberger (Menger) for every $n < \omega$. 
\end{theorem}

Then it would suffice to take $X=2^\omega$ in the above theorem. The Cantor set is, in fact, known to be Menger, but not Rothberger. Indeed, any Rothberger subset of the reals has strong measure zero.

We would like to show an alternative, more combinatorial, construction of an $M$-separable non-$R$-separable space. We will use this example later to answer a question of Hutchison.

\begin{example} \label{countex}
A countable $M$-separable non-$R$-separable space $X$.
\end{example}

\begin{proof}
Let $X=Fn(\omega, \omega;\omega)$, that is the set of all finite partial functions from $\omega$ to $\omega$. Provide $X$ with the following topology. A basic neighborhood of the point $F \in X$ is a set of the form $$V(F, \mathcal{F})=\{G \in X: G \supset F \wedge (\forall f \in \mathcal{F})(\forall n \in \dom{G} \setminus \dom{F})(G(n) \neq f(n))\}$$
where $\mathcal{F} \in [\omega^\omega]^{<\omega}$.

\newClaim
\begin{Claim}
$X$ is not $R$-separable. 
\end{Claim}

\begin{proof}[Proof of Claim 1] Let $D_n=\{F \in X: n \in \dom{F}\}$. Then $D_n$ is dense in $X$. Suppose by contradiction that we can find points $F_n \in D_n$ such that $D=\{F_n: n < \omega \}$ is dense in $X$. Let $f \in \omega^\omega$ be the function defined by $f(n):=F_n(n)$. Then $V(\emptyset, \{f\})$ is easily seen to miss $D$.
\renewcommand{\qedsymbol}{$\triangle$}
\end{proof}

\begin{Claim}
  Let $D \subset X$ and $k <\omega$ be such that:
\begin{equation} \label{maincond}
D \cap V(\emptyset, \mc F) \neq \emptyset \text{ for each } \mathcal{F} \in [\omega^\omega]^k.
\end{equation}
Then there is a finite subset $D'$ of $D$ such that :
$$D' \cap V(\emptyset, \mathcal{F}) \neq \emptyset \text{ for each } \mathcal{F} \in [\omega^\omega]^k.$$
\end{Claim}

\begin{proof}[Proof of Claim 2]
For $F \in X$ let 
$$W(F)=\{(f_1, \dots f_k) \in (\omega^\omega)^k: \forall 1 \leq j \leq k\ (\forall i <|F|)\ f_j(i) \neq F(i) \}.$$ 
Let $T$ be the cofinite topology on $\omega$. Then $T$ is compact, and $W(F)$ is an open subset of the compact space $(T^k)^\omega$.  

By ($\ref{maincond}$), the set $\{W(F): F \in D \}$ is an open cover of $(T^\omega)^k$. So there is a finite set $D' \subset D$ such that $\{W(F): F \in D' \}$ covers $(T^\omega)^k$. Then $D'$ satisfies the requirements of the Claim.
\renewcommand{\qedsymbol}{$\triangle$}
\end{proof}

\begin{Claim}
 $X$ is $M$-separable.
\end{Claim}

\begin{proof}[Proof of Claim 3] 
Enumerate $X \times \omega$ as $\{(F_n, k_n): n < \omega \}$. Using Claim 2 we choose, for each $n<\omega$, a finite subset $D'_n$ of the set $\{F \in D_n: F_n \subset F \}$ such that 
$$D'_n \cap V(F_n, \mathcal{F}) \neq \emptyset \text{ for each } \mathcal{F} \in [\omega^\omega]^{k_n}.$$
Then $D=\bigcup \{D'_n: n \in \omega \}$ is dense.  Indeed, if $F\in X$ and $\mc F\in \br \omega^\omega;<\omega;$, then pick $n\in \omega$ with $F_n=F$ and $k_n=k$.
Then there is $d\in D'_n\subs D$ with $d\in V(F,\mc F)$.
\renewcommand{\qedsymbol}{$\triangle$}
\end{proof}
\end{proof}

In \cite{GNP}, Gruenhage, Natkaniec and Piotrowski say that a space $X$ satisfies property $(GC)$ if there is a disjoint, countable collection $\mathcal{N}$ of nowhere dense sets such that, for every non-empty open set $U \subset X$ we have $|\{N \in \mathcal{N}: U \cap N = \emptyset \}|<\omega$. In her PhD thesis \cite{Hutch} Hutchison proves that every space having a dense metrizable subset satisfies $(GC)$ and that property $(GC)$ is equivalent to having a $\sigma$-disjoint $\pi$-base in the realm of linearly ordered topological spaces. This notion is strictly intertwined with the notion of \emph{groupable} dense set, which is the basis for another selective version of separability.

\begin{definition}
A dense set $D \subset X$ is called \emph{groupable} if it admits a partition $\mathcal{A}=\{A_n: n < \omega \}$ into finite sets such that every open subset of $X$ meets all but finitely elements of $\mathcal{A}$.
\end{definition}

\begin{definition}
A topological space $X$ is called $(GN)$-separable (from Gerlits and Nagy) if for every sequence $\{D_n: n < \omega \}$ of dense sets there are points $d_n \in D_n$ such that $\{d_n: n < \omega \}$ is a groupable dense set.

We say that $X$ is {\em H-separable} if for each sequence $\{D_n: n < \omega \}$ of dense sets, one can pick finite sets  $F_n \subs D_n$ so that for every nonempty open set $O \subset X$, the intersection $O \cap F_n$ is nonempty for all
but finitely many $n$.
\end{definition}

$GN$-separability was introduced by Di Maio, Ko\v cinac and Meccariello \cite{DKM} under the name of \emph{selection principle $S_1(\mathcal{D}, \mathcal{D}^{gp})$} while $H$-separability was introduced by Bella, Bonanzinga and Matveev in \cite{BBM}.

Clearly, every space having a groupable dense set satisfies property $(GC)$, and actually, a space has a groupable dense set if and only if it satisfies property $(GC)$ witnessed by a collection of finite sets. 
Hutchison asked if one could add $(GN)$-separability to this equivalence, for the class of countable spaces. As a partial result, she noted that in a space satisfying $(GC)$ witnessed by a collection of finite sets every dense set is groupable. 
We are going to give a negative answer to her question.
Actually, we can prove a bit more (see Figure 1 on the relationship the properties we defined above).


\begin{figure}[h!]

\begin{tikzpicture}
[pro/.style={ inner sep=4pt,minimum size=4mm,draw=white},back line/.style={densely dotted},
        normal line/.style={-stealth},
        cross line/.style={normal line,
           preaction={draw=white, -, 
           line width=6pt}},
]

\matrix[row sep=10mm,column sep=15mm]{

\snode {pomega}{${\pi w}(X)={\omega}$};&\snode {Hsep}{H-separable}; \\

\snode {GNsep}{GN-separable}; & 
;& 
\snode {edegru}{$\exists$ dense groupable};\\

\snode {R-separable}{R-separable}; 
& \snode {M-separable}{M-separable};& \snode {separable}{separable};\\
};

\path[->]

  (pomega) edge[thick] (GNsep) 
  (pomega) edge[thick] (Hsep)

  (GNsep) edge[thick] (R-separable) 

  (GNsep) edge[thick] (edegru) 
  (R-separable) edge[thick] (M-separable)
  (M-separable) edge[thick] (separable)

(edegru)  edge[thick] (separable)
(Hsep) edge[cross line]  (M-separable)
(Hsep) edge[thick]  (edegru)

;

\end{tikzpicture}
\caption{}
\end{figure}



\begin{theorem}
There is a  countable, H-separable, non-$R$-separable space $X$.
\end{theorem}

\begin{proof}
Let $X$ be the space from Example $\ref{countex}$. 

Assume that   $\{D_n: n < \omega \}$ is a sequence of dense sets.

Enumerate $X \times \omega$ as $\{(F_m, k_m): m < \omega \}$. Using Claim 2  from Example $\ref{countex}$ we choose, for each $n<\omega$ and $m\le n$  , a finite subset $D^n_m$ of the set $\{F \in D_n: F_m \subset F \}$ such that 
$$D^n_m \cap V(F_m, \mathcal{F}) \neq \emptyset \text{ for each } \mathcal{F} \in [\omega^\omega]^{k_m}.$$
Let $D_n=\bigcup\{D^n_m:m\le n\}$.

If $(F,\mc F)$ is a basic open set, $F_m=F$ and $k_m=|\mc F|$, then 
$(F,\mc F)\cap D^n_m\ne \empt$ for $m\le n$, and so 
$(F,\mc F)\cap D_m\ne \empt.$

Thus $X$ is $H$-separable.
\end{proof}

\begin{definition}
We say that a  set $X\subs 2^{\omega_1}$ is  a {\em very strong $HFC$} iff
for each sequence $\{A_n:n\in \omega\}$ of pairwise disjoint, non-empty finite subsets of $X$
there is $\beta<\omega_1$ such that for all $s\in Fn(\omega_1\setm \beta, 2;\omega;\omega)$
there are infinitely many $n$ with $A_n\subs [s]$, where
$[s]=\{x\in  2^{\omega_1}:s\subs x\}$.  
\end{definition}

The following proposition is straightforward.
\begin{proposition}
A  very strong $HFC$ can not contain a groupable dense set.
\end{proposition}

\begin{theorem}\label{tm:Rsepnotgr}
In $V^{Fn(\omega\times\omega_1,2;\omega)}$
there is a countable $R$-separable space without  a groupable dense subset. 
\end{theorem}

\begin{proof}
If $\mc G$ is the generic filter, then $g=\cup\mc G$ is a function from $\omega\times \omega_1 $ to $2$.
If you define, for $n\in \omega$, the function $x_n\in {}^{\omega_1}2$ by the formula
$x_n(\alpha)=g(n,\alpha)$, then standard density arguments give that $X=\{x_n:n\in \omega\}$ is  a very strong $HFC$ which is dense in $2^{\omega_1}$.
\end{proof}

 In \cite[Example 3.2]{GS} Gruenhage and Sakai   constructed a  maximal R-separable space from CH, and it is straightforward  that such a space 
also satisfies the requirements of theorem \ref{tm:Rsepnotgr}.

\begin{question}
(1) Is there a ZFC example of an crowded R-separable space without a groupable dense subset?\\
(2) Is there a ZFC example of  a (countable) GN-separable space with uncountable
$\pi$-weight?  
\end{question}

To get a consistent example for   Question 3.10 (2), any counterexample to Malykhin's problem would do. Indeed, any countable Frechet-Uryson non-metrizable group has uncountable $\pi$-weight 
 and every countable Frechet-Urysohn space without isolated points is GN-separable, by \cite{GS}.

Our next goal is to analyze selective versions of $d$-separability and $nwd$-separability.

\begin{definition}
A space $X$ is called \emph{$D$-separable} (respectively, {\em $NWD$-separable}) if for every sequence $\{D_n: n < \omega \}$ of dense sets there are discrete (respectively, nowhere dense) sets $E_n \subset D_n$ such that $\bigcup_{n<\omega} E_n$ is dense in $X$.
\end{definition}

Observe that if $X$ is $D$-separable (respectively, $NWD$-separable) then every dense subset of $X$ is $D$-separable (respectively, $NWD$-separable) as well. Also, if $X$ is $D$-separable (respectively, $NWD$-separable) then it is also $d$-separable (respectively, $nwd$-separable).

$D$-separability was already investigated in \cite{GS} and \cite{AJR}. Let us now introduce a general framework for dealing with selection principles for dense sets.

\begin{definition}\label{gendef}
For each topological space $X$, let
$\mbb A_X\subs \mc P(X)$.
We say that $X$ is {\em $\mbb A$-separable} iff
for each sequence $\{D_n:n\in {\omega}\}$ of dense subsets of $X$
there are $A_n\in  \mc P(D_n)\cap \mbb A_X$ for $n\in {\omega}$
such that $\bigcup\{A_n:n\in {\omega}\}$ is dense in $X$. 
\end{definition}

The formulation and the proof of the following result is based on \cite[Theorem 2.2.]{GS}.
\begin{theorem}\label{tm:union}
Assume that  for each topological space $X$ we have 
$\mbb A_X\subs \mc P(X)$ such that 
\begin{enumerate}[(a)]
 \item $\mbb A_X$ is an ideal,
 \item if $Z\subs X$, then $\mbb A_Z\subs \mbb A_X$,
 \item if $U\subs X$ is open, then $\mbb A_U=\{A\cap U:A\in \mbb A_X\}$.
\end{enumerate}
Then the union of two $\mbb A$-separable spaces is $\mbb A$-separable.
\end{theorem}

%


%
%


\begin{proof}
First we need some easy observations.

\begin{obs}\label{obs:open}
If $X$ is $\mbb A$-separable, and $U\subs X$ is open, then 
$U$ is also $\mbb A$-separable. 
\end{obs}

Indeed, if $\{D_n:n\in {\omega}\}$ are dense subset of $U$, then
$E_n=D_n\cup (X\setm \overline U)$ are dense subsets of $X$
for $n\in {\omega}$, so there are sets $A_n\in \mbb A$
with $A_n\subs D_n\cup (X\setm \overline U)$ such that 
$A=\bigcup\{A_n:n\in {\omega}\}$ is dense in $X$.
Let  $B_n=A_n\cap D_n$. Observe that $B_n\in \mbb A_U$ by (c).
Since   $\bigcup\{B_n:n\in {\omega}\}=A\cap U$, the set  $\bigcup\{B_n:n\in {\omega}\}$ is dense in $U$,
which proves the Observation.

We need the following lemma which corresponds to \cite[Lemma 2.1]{GS}.
\begin{lemma}\label{lm:union}
A topological space $X$ is $\mbb A$-separable iff for every decreasing 
sequence $\{ D_n : n \in  \omega\}$ of dense subsets of $X$, there are
sets $E_n \subs  D_n$ from  $\mbb A_X$ for  $n \in \omega $ such that 
$\bigcup\{ E_n : n \in \omega\}$ is dense in $X$.
\end{lemma}

\begin{proof}
Let $\{ C_m : m \in \omega\}$ be a sequence of dense subsets of $X$ . 
For each $n \in \omega $, let
  $D_n = \bigcup\{C_m:m\ge n\}$ . 
The sequence $\{D_n:n\in {\omega}\}$ is decreasing, so 
there are  sets 
$E_n \subs D_n$ from  $\mbb A_X$ for $n \in \omega $ such that 
$\bigcup\{ E_n : n \in \omega\}$ is dense in X . 
Let $F_m= C_m\cap \bigcup\{E_n:n\le m\}$.
Then $F_m\subs C_m$ and $F_m\in \mbb A_X$ by (a), and 
$\bigcup\{F_m:m<{\omega}\}=\bigcup\{E_n:n<{\omega}\}$, so
$\bigcup\{F_m:m<{\omega}\}$ is dense.
\end{proof}

Assume  $X=Y \cup Z$,  where $Y$ and $Z$ are $\mbb A$-separable.
Assume that 
$\{ D_n : n \in {\omega}\}$ is  a 
sequence of dense subsets of $X$ . 
By Lemma \ref{lm:union} we can assume that the sequence is decreasing.

Put
 $U_n = X \setminus \overline {Y \cap D_n}$ . 
Then $\{U_n : n \in  {\omega}\}$ is an increasing family of open sets in $X$. 

Fix an $n \in \omega$. Clearly $Z\cap U_n$ is dense in $U_n$.
Since $U_n$ is open, the subspace $U_n \cap Z$  of $Z$ is  $\mbb A$-separable.

For $k\ge n$, 
$U_n \cap D_k \subs  U_n\cap D_n \subs  U_n \cap Z \subs U_n$, so 
the set    $U_n \cap  D_k$
is  dense in $Z\cap U_n$.   Since $Z\cap U_n$ is $\mbb A$-separable,  
there are  sets $F_{n,k} \subs U_n \cap D_k$ from $\mbb A_{Z\cap U_n}$
for $k\ge n$ such that 
$\bigcup\{ F_{n,k} : k\ge n\}$ is dense in $U_n$ . 

Since $\mbb A_{Z\cap U_n}\subs \mbb A_Z$ by (b), we have
$\{F_{n,k}:n\le k<{\omega}\}\subs \mbb A_Z$.

Now put  $F_k = \{ F_{n,k} : n\le k\}$
for $k \in \omega$.  Then $F_k\subs D_k$ and   $F_k\in \mbb A_Z$ by (a) for $k<{\omega}$,
and
$\bigcup\{ F_k : k \in \omega\}$ is dense in $\bigcup\{U_n : n \in \omega\}$. 

Let $V=X\setm \overline{\bigcup\{U_n : n \in \omega\}}$.
For each $n\in {\omega}$,
$D_n\cap V\subs D_n\setm U_n\subs \overline{D_n\cap Y}$,
so $D_n\cap V\cap Y$ is dense in $V$. 

Since $Y\cap V$ is $\mbb A$-separable by observation \ref{obs:open}, 
there are  sets  $G_n \subs Y\cap V \cap D_n$
with $G_n\in \mbb A_{Y\cap V}$
  for $n \in \omega $ 
such that 
$\bigcup\{G_n : n \in \omega\}$ is  dense in $Y\cap V$, and so it is also dense in $V$.
Since $\mbb A_Z\cup \mbb A_Y\subs  \mbb A_X$ by (b),
we have $F_n\cup G_n\in \mbb A_X$ by (a).
Thus   $F_n \cup G_n$ is a  subset of $D_n$  from $\mbb A_X$ and 
$\bigcup\{ F_n \cup G_n : n \in \omega\}$ is dense in $X$. 
\end{proof}

\begin{corollary}\label{cor:union}
\begin{enumerate}
\item \cite{BMS} The union of two D-separable spaces is D-separable.
\item The union of two NWD-separable spaces is NWD-separable. 
\end{enumerate}
\end{corollary}

\section{Examples in ZFC and related results} \label{zfc}

The first part of this section deals with the construction of an $NWD$-separable space which is not $d$-separable.

\newcommand{\nice}{deep}

Given a cardinal $\kappa$, 
we say that a $\pi$-base $\mc U$  is {\em$\kappa$-\nice} iff
for each  decreasing sequence  $\{U_n\}_{n\in \kappa}\subs \mc U$ we have $int(\bigcap_{n\in \kappa}U_n)\ne \empt$.


 \begin{lemma} \label{lemmameag}
 For each infinite cardinal $\kappa$, there is a crowded regular space of size $2^\kappa$ which has a $\kappa$-\nice\ $\pi$-base.
 \end{lemma}

\begin{proof}
We claim that a space with the claimed properties is $$X=\Sigma_\kappa(2^{\kappa^+})=\{f\in 2^{\kappa^+}: |f^{-1}\{1\}|\le \kappa \},$$ 
endowed with the $\kappa$ supported box product topology. 

If $s\in 2^{<{\kappa^+}}$,  then let
$B(s)=\{f\in X: s\subs f\}$.
Put
\begin{equation}\notag 
\mc U=\{B(s): s\in 2^{<{\kappa^+}}\}.     
\end{equation}
Then $\mc U$ is  actually a base of $X$.
To show that $\mc U$ is \nice, assume that  
$\<B(s_n):{n\in \omega}\>$ is decreasing. Then we have $s_0\subs s_1\subs \dots$, and so 
$s=\bigcup_{n\in \kappa}s_n$ is a function, and $B(s)\subs \bigcap_{n\in \kappa}B(s_n)$. 
\end{proof}


 \begin{lemma}\label{lm:NWD-in-product}
 Assume that  $X$ has a $\omega$-\nice\ $\pi$-base. Then $Y=X\times \mbb Q$ is NWD-separable. 
 \end{lemma}

\begin{proof} We need two claims.
\begin{claim}
If $S\subs Y$ is dense, then 
for each non-empty open $U\subs X$ and $p<q\in \mbb Q$
there is non-empty open $V\subs U$ and $p<r<q\in \mbb Q$
such that $$\pi_r(S)\stackrel{def}{=}\{x\in X: \<x,r\>\in S\}$$ is dense in $V$. 
\end{claim}
\begin{proof}[Proof of the claim]
Assume on the contrary that  the sets $\pi_r(S)$ are nowhere dense. 
Enumerate $(p,q)\cap \mbb Q$ as $\{r_n:n\in {\omega}\}$.
Construct a decreasing sequence  $\{U_n\}_{n\in {\omega}}\subs \mc U$
such that $U_n\cap \pi_{r_n}(S)=\empt$.
Then the set  $W=int(\bigcap_{n\in {\omega}} W_{r_n})$ is non-empty.
Thus  $S\cap (W\times (p,q))=\empt$, contradiction.
\end{proof}

\begin{claim}
If $\{S_n:n<{\omega}\}\subs Y$ are dense, then 
for each a non-empty open $U\subs X$ there is non-empty open $V\subs U$ and 
there is a sequence  $\{r^V_n:n<{\omega}\}\subs \mbb Q$
such that 
\begin{equation}\label{eq:snrnv}
\text{$\bigcup_{n\in \omega}\{S_n\cap (X\times \{r^V_n\})\}$ is dense in $V\times \mbb Q$.
}
\end{equation}
\end{claim}
\begin{proof}[Proof of the claim]
Enumerate the pairs 
$\{\<p,q\>:p,q\in \mbb Q, p<q\}$ as $\{\<p_n,q_n\>:n<{\omega}\}$.

Construct  a decreasing sequence of open sets $U_0\supset U_1\supset U_1\supset \dots$
from $\mc U$
and distinct rational numbers $r_n$ such that   
\begin{enumerate}
\item $U_0\subs U$,
 \item $r_n\in (p_n,q_n)$,
\item  {$\pi_{r_n}(S_n)$ is dense in $U_{n+1}$.
}
\end{enumerate}
The construction can be carried out by the previous claim.
Then  $V=int(\bigcap_{n\in {\omega}}U_n)\ne \empt$ works
if we take $r^V_n=r_n$.
\end{proof}
Let $\mc V$ be a maximal disjoint family of open sets in $X$
such that every $V\in \mc V$ satisfies the requirements of the previous claim.
Then $\bigcup\mc V$ is dense in $X$ by the previous claim.

Let 
\begin{equation}\label{eq:andef}
 A_n=\bigcup_{V\in \mc V} S_n\cap (V\times \{r^V_n\}).
\end{equation}
Then $A_n$ is nowhere dense because for $V\in \mc V$ we have
$A_n\cap (V\times \mbb Q)\subs X\times \{r^V_n\}$.
Moreover $A=\bigcup_{n\in \omega}A_n$ is dense, because $A\cap (V\times \mbb Q) $ is dense in $V\times \mbb Q$
by  (\ref{eq:snrnv}) and (\ref{eq:andef}).
\end{proof}

\begin{example}\label{ex:NWD_not_d}
There is an NWD-separable, but not d-separable space of size $\mathfrak{c}$.
\end{example}
\begin{proof}
We show that if  $Y=X\times \mbb Q$ where
 $X$ is the space from Lemma $\ref{lemmameag}$, then $Y$ is NWD-separable, but not d-separable.

By Lemma $\ref{lm:NWD-in-product}$ the space $Y$ is NWD-separable.

If $\{D_n:n\in {\omega}\}$ are discrete in $Y$, then 
\begin{equation}
\pi_q(D_n)\stackrel{def}{=}\{x\in X: \<x,q\>\in D_n\} 
\end{equation}
is  discrete, and so nowhere dense in $X$ for  $q\in \mbb Q, n\in {\omega}$. 
Since $X$ has a $\omega$-\nice\ $\pi$-base, 
there is an open $U\subs X$ with $(U\times \mbb Q)\cap \bigcup_{n\in {\omega}}D_n=\empt$,
and so $\bigcup_{n\in {\omega}}D_n$ is not dense. 
\end{proof}

In Theorem \ref{ex:unified} we show that it is consistent that $2^\omega$ is large, but there is  a $NWD$-separable non-$D$-separable space of size $\aleph_1$.

\begin{question} \label{smallcount}
Is there a $NWD$-separable non-$D$-separable space of size $\aleph_1$ in ZFC? Is there at least one whose size is bounded in ZFC?
\end{question}

At least the first question seems to require techniques different from those of this paper. Indeed,  a space having a  $\omega$-\nice\ subbase is Baire, and Shelah and Todorcevic \cite{ST} showed modulo the consistency of an inaccessible cardinal, 
that the existence of a Baire space of size $\aleph_1$ is independent from ZFC. 

Let us continue with another example: a countable, not $NWD$-separable space. The following result was proved in \cite{BMS} using a direct construction; the $\mathcal{D}$-forced technology of \cite{JSS} can be used to give an alternative proof.

\begin{example}\label{ex:countable_not_NWD}
$2^{\mf c}$ has a countable, dense, not $NWD$-separable subspace. 
\end{example}

\begin{proof}
By \cite[Theorem 4.9]{JSS}
there is a countable, dense nodec subspace $X$
of $2^{\mf c}$ such that $X$ can be partitioned into 
submaximal dense subspaces $\mc D=\{D_n:n\in {\omega}\}$,
and $X$ is   $\mc D$-forced, i.e. if $D\subs X$ is somewhere dense, then 
$D\supset D_n\cap U$ for some $n\in {\omega}$ and non-empty open set $U$.

Then $X$ is not NWD-separable.
Indeed, if $E_n\subs D_n$ is nowhere dense, then $E=\bigcup_{n\in {\omega}}E_n$
is not dense, because it can not contain any $D_n\cap U$. 
\end{proof}

\begin{example}\label{ex:nwd_not_d_not_NWD}
There is an nwd-separable, but not d-separable and not NWD-separable 
space of size $2^{2^\mathfrak{c}}$.
\end{example}

\begin{proof}
By Lemma \ref{lemmameag}, there is a crowded regular space $X$ of size 
$2^{2^{\mathfrak{c}}}$ which has a 
$2^\mathfrak{c}$-\nice\ $\pi$-base. 
Let $Y=2^\mathfrak{c}$ with the product topology.

We claim that $X\times Y$ has the required properties.
\begin{claim}
 $X\times Y$ is nwd-separable.
 \end{claim}

Indeed, let $D=\{d_n:n\in \omega\}$ be dense in $Y$.
Then $S_n=X\times\{d_n\}$ is nowhere dense in $X\times Y$,
but $\bigcup_{n\in \omega }S_n=X\times D$ is dense in $X\times Y$.

\begin{claim}
 $X\times Y$ is not d-separable.
 \end{claim}

If $\{D_n:n\in {\omega}\}$ are discrete in $Y$, then 
\begin{equation}
\pi_y(D_n)\stackrel{def}{=}\{x\in X: \<x,y\>\in D_n\} 
\end{equation}
is  discrete, and so nowhere dense in $X$ for  
$y\in Y, n\in {\omega}$. 
Since $X$ has a $|Y|$-\nice\ $\pi$-base, 
there is an open $U\subs X$ with 
$(U\times Y)\cap \bigcup_{n\in {\omega}}D_n=\empt$,
and so $\bigcup_{n\in {\omega}}D_n$ is not dense. 

\begin{claim}
 $X\times Y$ is not NWD-separable.
 \end{claim}

By \cite[Theorem 4.9]{JSS}
there is a countable, dense nodec subspace $T$
of $2^{\mf c}$ such that $T$ can be partitioned into 
submaximal dense subspaces $\mc D=\{D_n:n\in {\omega}\}$,
and $T$ is   $\mc D$-forced, i.e. if $D\subs T$ is somewhere dense, then 
$D\supset D_n\cap V$ for some $n\in {\omega}$ and non-empty open set $V$.

Let $E_n=X\times D_n$ for $n\in \omega$.
We show that if $F_n\subs E_n$ is nowhere dense, then 
$F=\bigcup_{n\in \omega}F_n$ cannot be dense.

Let $\{B_i:i<\mathfrak{c}\}$ be a base of $Y$, and 
fix a $\mathfrak{c}$-\nice\ $\pi$-base $\mc U$ of $X$.

By induction on $n$ construct
a decreasing sequence $\{U_n:-1\le n<\omega \}\subs \mc U$
as follows. Let $U_{-1} \in\mc U$ be arbitrary nonempty.
Assume that $U_{n-1}$ is constructed.

By transfinite induction construct a decreasing sequence 
$\{U^n_i:i<\mathfrak{c}\}\subs \mc U\cap \mc P(U_{n-1})$
such that for all  $i<\mathfrak{c}$ there is a non-empty
$V^n_i\subs B_i$ such that 
$$F_n\cap (U^n_i\times V^n_i)=\empt.$$
Since $\mc U$ is $\mathfrak{c}$-\nice,
 there is a nonempty open $U_n\subs X$ such that $U\subs U^n_i$
for all   $i<\mathfrak{c}$. 

Finally there is $U\in \mc U$ such that $U\subs U_n$
for all $n\in \omega$.

Since $U\subs U^n_i$, we have $F_n\cap (U\times V^n_i)=\empt$.
Write 
$V_n=\bigcup_{i<\mathfrak{c}}V^n_i$ and $G_n=D_n\setm V_n$  for $n<\omega$.
Then $V_n$ is open dense in $Y$ and   $F_n\cap (U\times V_n )=\empt$, 
so 
\begin{equation}\label{gn}
 F_n\cap (U\times Y)\subs U\times G_n.
\end{equation}

Since $G=\bigcup _{n\in \omega}G_n$ cannot contain $V\cap D_n$
for a non-empty open $V$, 
so $G$ is nowhere dense  in $T$ because $T$ is $\mc D$-forced.
But $T$ is dense in $Y$, so $G$ is nowhere dense in $Y$.
Especially there is a nonempty open $V$ with $V\cap G=\empt$.

But then, by (\ref{gn}), 
\begin{multline}
F\cap (U\times V)=\bigcup_{n\in \omega }F_n\cap (U\times V)= 
\bigcup_{n\in \omega } 
\big(F_n\cap (U\times Y)\big )\cap (U\times V)\subs\\ 
\bigcup_{n\in \omega} (U\times G_n) \cap (U\times V)=
(U\times G)\cap (U\times V)= U\times (G\cap V)
=\empt. 
\end{multline}
so $F$ cannot be dense, which was to be proved.
\end{proof}

Motivated by Example \ref{ex:countable_not_NWD}, 
one can define the following cardinal invariants:

\begin{definition} \cite{BMS}
Let 
\begin{equation}
\mf{ds}=\min \{{\kappa}: 2^{\kappa}\ \text{ is not $D$-separable}\}
\end{equation}
\begin{equation}\notag
\mf{cds}=\min\{{\kappa} : 2^{\kappa}\
\text{ contains a countable non-D-separable subspace}\}. 
\end{equation}
\end{definition}

We have $\mf{cds}\le 2^{\omega}$ by Example \ref{ex:countable_not_NWD}. Moreover, as $$\mathfrak{d}=\min\{\kappa: 2^\kappa \text{ contains a countable non-M-separable dense subspace}\}$$ was shown in  \cite{BBM}, we also have $\mathfrak{d} \leq \mathfrak{cds}$.

In \cite{BMS} the authors proved that the space $X^{2^{d(X)}}$ is never D-separable. In particular, if $X$ is separable, then $X^{2^{\omega}}$ is not D-separable. This exponent appears far from optimal and we can in fact improve it for separable spaces; the next theorem also solves Question 44 from \cite{BMS}, while we note that an alternative solution to this question was provided in \cite{AJR}.

\begin{theorem}\label{Lsp}
If $X$ is an  separable space with $|X|\ge 2$ then some dense subspace $Y$ of $X^{\oo}$ is not d-separable; hence
$X^{\oo}$ is not $D$-separable for any  separable $X$ with $|X|\ge 2$. Hence $\mf{ds}=\omega_1$.
\end{theorem}

\begin{proof}
J. Moore in \cite{M} constructed an L-space 
$L=\{f_{\alpha}:{\alpha}<\oo\}\subs {\omega}^\oo$ such that 
\begin{equation}\label{eq:dense}
\text{
$\big|L\cap [\varepsilon]\big|=\oo$ for each finite function $\varepsilon\in Fn(\oo,{\omega};{\omega})$.
} 
\end{equation}

Let $D=\{d_n:n<{\omega}\}$ be dense in $X$. 

For ${\alpha}<\oo$ define $y_{\alpha}\in  D^\oo$ as follows:
\begin{equation}\label{eq:ya}
y_{\alpha}({\beta})=\left \{
\begin{array}{ll}
d_{f_{\alpha}({\beta})}&\text{if ${\beta}<{\alpha}$,}\\ 
d_0&\text{if ${\beta}\ge {\alpha}$.}
\end{array}
\right . 
\end{equation}
Let $Y=\{y_{\alpha}:\alpha<\omega_1\}$.

Then $Y$ is dense in $X^\oo$ by (\ref{eq:dense}).
Moreover, if $D\subs Y$ is discrete, then $D$ is countable because
$L$ is hereditarily Lindel\"of and the map $f_\alpha \to y_\alpha$ is continuous.
So $Y$ is not d-separable, because $d(Y)=\oo$ by (\ref{eq:ya}).
\end{proof}

However, the following remains open:

\begin{conjecture}
 The space $X^{{d(X)^+}}$ is never $D$-separable.
\end{conjecture}


 The next corollary solves Question 45 from \cite{BMS}:

\begin{corollary}
$MA+\neg CH$ implies $\mf{ds}<\mathfrak{cds}$.
\end{corollary}

\begin{proof}
$\mathfrak{ds}=\aleph_1$ is true in every model of ZFC. Since $\mathfrak{d} \leq \mathfrak{cds}$ and $\mathfrak{d}=\mathfrak{c}$ in every model of MA the statement of the corollary follows from the failure of CH.
\end{proof}

The authors of \cite{BMS} ask \emph{what is $\mf{cds}$}? Up to this point it was even unknown whether $\mf{cds}$ could consistently be less than the continuum, so the following theorem may be considered a partial answer to Question 43 of \cite{BMS}.

\begin{theorem}\label{tm:cof_ds}
If $cof(\mc M)={\omega}_1$ then  $2^{\omega_1}$ has a countable dense subspace  
which is not NWD-separable. So $\mf{cds}={\omega}_1$.
\end{theorem}

We recall some definitions from \cite{JSS}.

Let $S$ be  a set, 
and
\begin{equation}\notag
{\mbb B}=\left\{\<B_{\zeta}^0,B_{\zeta}^1\>:{\zeta}<{\mu}
\right\} 
\end{equation}
 be a family of partitions of $S$.
We say that
${\mbb B}$ is {\em independent} iff
\begin{displaymath}
{\mbb B}[\varepsilon]\stackrel{def}=
\bigcap\{B^{{\varepsilon}({\zeta})}_{\zeta}:{\zeta}\in\dom
{\varepsilon} \}\ne\empt
\end{displaymath}
for each ${\varepsilon}\in Fn(\mu,2;\omega)$. 
${\mbb B}$ is {\em separating } iff for each $\{{\alpha},{\beta}\}\in \br S
;2;$ there are ${\zeta}<\mu$ and ${\rho}\ne {\nu}<2
$ such that ${\alpha}\in B^{\rho}_{\zeta}$
and ${\beta}\in B^{\nu}_{\zeta}$.

We shall denote by $\tau_{\mbb B}$ the (obviously zero-dimensional)
topology on $S$ generated by the subbase $\{B_{\zeta}^0,B_{\zeta}^1: \zeta <
\mu\}$, moreover we set
$X_{{\mbb B}}=\<{S},{\tau}_{{\mbb B}}\>$. Clearly, the family
$\{{\mbb B}[{\varepsilon}]:{\varepsilon}\in Fn(\mu,2;\omega)\}$
is a base for the space $X_{{\mbb B}}$. Note that $X_{{\mbb B}}$ is
Hausdorff iff ${\mbb B}$ is separating.

\begin{obs}\label{lm:embed}
Let ${\lambda}$ be an infinite cardinals. Then, up to
homeomorphisms, there is a natural one-to-one correspondence
between countable dense subspaces $X$ of ${D(2)}^{\lambda}$  and  spaces of the form 
$X_{{\mbb B}} =
\langle\omega,{\tau}_{{\mbb B}}\rangle$, where ${\mbb B} = \{\langle B_{\xi}^0,B_{\xi}^1\rangle
: {\xi}<\lambda\}$ is a separating and independent
family of $2$-partitions of $\omega$. 
\end{obs}

\begin{proof}[Proof of Theorem \ref{tm:cof_ds}]
By the Hewitt-Marczewski-Pondiczery Theorem, \cite[Theorem
2.3.15]{En},
there are partitions 
\begin{equation}\notag
 \left\{\<F^j:j<{\omega}\>\right\}\cup\left\{\<B_{\zeta}^i:i<2\>:{\zeta}<{\omega}_1
\right\}
\end{equation}
of ${\omega}$ such that
\begin{equation}\notag
 {\mbb B}=\left\{\<B_{\zeta}^i:i<2\>:{\zeta}<{\omega}_1
\right\} 
 \end{equation}
is separating and 
\begin{equation}\label{FB}
F_j\cap  {\mbb B}[{\varepsilon}]\ne\empt 
\end{equation}
for all $j\in {\omega}$ and 
${\varepsilon}\in Fn({\omega_1},2;\omega;\omega)$,
We can assume that 
\begin{equation}
\forall x\ne y\in {\omega}\ \exists n<{\omega} 
\ (x\in B^0_n \land y\in B^1_n.)
\end{equation}

Next, let us fix any partition $\{I_{\nu}:{\omega}\le {\nu}<{\omega}_1\}$ 
of ${\omega}_1\setm {\omega}$ into uncountable pieces with 
${\nu}\cap I_{\nu}=\empt$ and then by transfinite recursion on
${\omega}\le {\nu}<{\omega}_1$ define
\begin{itemize}
\item sequences $\<A^k_{\alpha}:k<{\omega}\>$ for ${\alpha}\in I_{\nu}$,
\item partitions $
\<C^0_\nu, C^1_\nu\>$ of ${\omega}$,
\end{itemize}
such that the inductive hypothesis
\begin{equation}\tag{$\phi_{\nu}$}
\forall\varepsilon \in Fn({\omega_1},2;\omega)\ 
\forall j<{\omega} 
|F^j\cap {\mbb B}_{{\nu}}[\varepsilon]|={\omega}
\end{equation}
holds,
where
\begin{displaymath}
{\mbb B}_{\nu}=\Bigl\{\<C^0_{\sigma},C^1_{\sigma}\>:{\omega}\le{\sigma}<\nu\Bigr\}
\cup \Bigl\{\<B^0_{\sigma},B^1_{\sigma}\>:  {\sigma}\in 
{\omega}\cup ({\omega_1}\setm \nu)\Bigr\}.
\end{displaymath}

Note that $(\phi_{\nu})$ simply says that every set $F^j$
is dense in the space $X_{{\mbb B}_{\nu}}$. We shall then
conclude that $\mbb C={\mbb B}_{{\omega_1}}$ is as required.

Let us observe first that $(\phi_{\omega})$ holds  because (\ref{FB}) holds and 
$
{\mbb B}[{\varepsilon}]={\mbb B}_{\omega}[{\varepsilon}]$.

Clearly, if ${\nu}$ is a limit ordinal and $(\phi_{\zeta})$ holds for
each ${\zeta}<{\nu}$ then $(\phi_{\nu})$ also holds. So the induction
hypothesis is preserved in limit steps.

Now consider a $\nu<\omega_1$ and assume that $(\phi_{\nu})$
holds.

Let 
\begin{equation}
C^0_{\nu}=B^0_{\nu}\cup\bigcup_{j<{\omega}}A^j_{\nu}; \quad
C^1_{\nu}=B^1_{\nu}\setm\bigcup_{j<{\omega}}A^j_{\nu}; 
\end{equation}
and
let
\begin{displaymath}
{\mbb B}'_{\nu}=\Bigl\{\<C^0_{\sigma},C^1_{\sigma}\>:{\omega}\le{\sigma}<\nu\Bigr\}
\cup \Bigl\{\<B^0_{\sigma},B^1_{\sigma}\>:  {\sigma}\in 
{\omega}\Bigr\},
\end{displaymath}
and consider the space $Y_{\nu}=\<{\omega},\tau_{\mbb B'_{\nu}}\>$.
Clearly $Y_{\nu}$ is homeomorphic to $\mbb Q$.

Let 
\begin{equation}
\mbb A_{\nu}=\{\<A_i:i<{\omega}\>: A_i\subs F_i, \text{$A_i$ is nowhere dense in
$Y_{\nu}$}\}. 
\end{equation}
If  $ A=\<A_i:i<{\omega}\>$ and $A'=\<A'_i:i<{\omega}\>$
are from $\mbb A_{\nu}$ let $A\prec A'$ iff $A_i\subseteq A'_i$ for 
each $i<{\omega}$.
 
Fremlin (see  \cite[Theorem 1.6]{BaHEHr}  ) proved that
\begin{equation}
cof(\mc M)=cof(nwd_{\mbb Q}), 
\end{equation}
where $nwd_{\mbb Q}$ is the family of nowhere dense subsets of $\mbb Q$.

Since the disjoint union of ${\omega}$ copies of $\mbb Q$ is homeomorphic to 
$\mbb Q$, we have 
\begin{equation}
cof(\mbb A_{\nu},\prec)=cof(nwd_{\mbb Q})=cof(\mc M)={\omega}_1. 
\end{equation}
Let $\{A_{\alpha}:{\alpha}\in I_{\nu}\}$ enumerate a cofinal subset of 
$\mbb A_{\nu}$.
Write $A_{\alpha}=\<A^i_{\alpha}:i<{\omega}\>$.

We  have to show that 
$(\phi_{{\nu}+1})$ holds. 

Assume, indirectly,
that for some $j<{\omega}$ and ${\varepsilon}\in Fn({\omega}_1,2;\omega)$
we have
\begin{equation}\notag
F_j\cap \mbb B_{{\nu}+1}[{\varepsilon}]=\empt.
\end{equation}
Fix  $\sigma\le\nu $
 with ${\nu}\in I_\sigma$.

Let ${\eta}=\varepsilon\restriction \sigma$.
Since $\<A^i_{\nu}:i<{\omega}\>\in \mbb A_\sigma$,
the set  $A^j_{\nu}$ was nowhere dense in $Y_\sigma$, i.e. there is
$\eta'\in Fn(\sigma,2;\omega)$, $\eta'\supseteq \eta$, such that 
\begin{equation}
\mbb B'_\sigma[\eta']\cap A^j_{\nu}=\empt.
\end{equation}
But 
\begin{equation}
\mbb B'_\sigma[\eta']= \mbb B_{\nu}[\eta']=\mbb B_{\nu+1}[\eta'],
\end{equation}
so 
\begin{equation}
(F^j\setm A^j_{\nu})\cap \mbb B_{\nu}[\eta'\cup\varepsilon]=
F^j\cap \mbb B_{\nu}[\eta'\cup\varepsilon]\ne\empt 
\end{equation}
by $(\phi_{{\nu}})$.
Thus 
\begin{equation}\notag
F^j\cap \mbb B_{{\nu}+1}[{\varepsilon}]\supset 
F^j\cap \mbb B_{{\nu}+1}[\eta'\cup {\varepsilon}]\supset
(F^j\setm A^j_{\nu})\cap \mbb B_{\nu}[\eta'\cup\varepsilon]\ne\empt,
\end{equation}
contradiction; the relation in the middle follows from the fact that $(F^j\setminus A^j_\nu)\cap C^i_\nu=(F^j\setminus A^j_\nu)\cap B^i_\nu$.

Finally we show that the sequence $\<F^j:j<{\omega}\>$ witnesses that 
$X_{\mbb C}$ is not NWD-separable.
Assume that $E_i\subs F^i$ is nowhere dense; being nowhere dense is witnessed by a dense open set which, in turn, is the countable union of basic open sets.
Thus there is $\sigma<\omega_1$ such that $E_i$ is nowhere dense in 
$Y_{\sigma}$ for all $i<\omega$. Then there is $\nu \in I_\sigma$ such that 
$E_i\subs A^i_{\nu}$ for $i<{\omega}$.

Then  $\bigcup\{E_i:i\in {\omega}\}\subs \bigcup\{A^i_{\nu}:i\in {\omega}\}
\subs C^0_{\nu}$, i.e. $\bigcup\{E_i:i\in {\omega}\}$ is not dense because
it does not intersect $C^1_\nu$.
\end{proof}

The following figure summarizes the (trivial) implications between  separation properties we considered in this section.
The labels of the arrows indicate the examples showing that the implications cannot be reversed.

\newcommand{\ssnode}[2]{\node (#1) [ pro] {\small #2};}

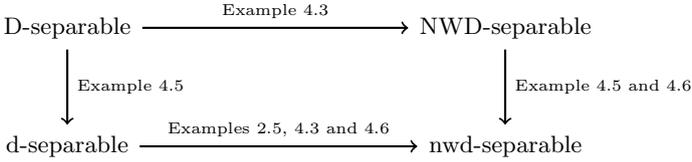
\begin{figure}[h!]

\begin{tikzpicture}
[pro/.style={ inner sep=4pt,minimum size=4mm,draw=white},back line/.style={densely dotted},
        normal line/.style={-stealth},
        cross line/.style={normal line,
           preaction={draw=white, -, 
           line width=6pt}},
]

\matrix[row sep=10mm,column sep=35mm]{

\ssnode {Dsep}{D-separable}; &  \ssnode {NWDsep}{NWD-separable};
\\

\ssnode {dsep}{d-separable}; &  \ssnode {nwdsep}{nwd-separable};
\\
};

\path[->]

  (Dsep) edge[thick]  node[auto] {\tiny Example \ref{ex:NWD_not_d} }   (NWDsep) 
  (Dsep) edge[thick]  node[auto] {\tiny Example  \ref{ex:countable_not_NWD} }  (dsep)

  (NWDsep) edge[thick]  node[auto] {\tiny Example  \ref{ex:countable_not_NWD} and \ref{ex:nwd_not_d_not_NWD}  }  (nwdsep) 
  (dsep) edge[thick]  node[auto] {\tiny Examples  \ref{ex:nwdnotd},  \ref{ex:NWD_not_d}  and   \ref{ex:nwd_not_d_not_NWD} }   (nwdsep)

;

\end{tikzpicture}

 \caption{Separation results in ZFC}

\end{figure}

We will get further consistency results in the next section, however following question remained open.
\begin{question}
 Is there a d-separable, NWD-separable, non D-separable space in ZFC?
\end{question}

\section{In the class of first-countable spaces - forcing counterexamples} \label{force}

In \cite{BD} Barman and Dow proved that every separable Fr\'echet space is $M$-separable. 
In \cite{GS} G. Gruenhage and M. Sakai observed  that separable Fr\'echet spaces are $R$-separable. The aim of this section is showing that no theorem of this kind can be proved in the context of $d$-separability ($nwd$-separability) 
and $D$-separability ($NWD$-separability), not even if one replaces Fr\'echet with first-countable. 

First, let us start with a lemma. 

\begin{lemma} \label{martin} Suppose that we fixed some ideal $\mbb A_X$ for each space $X$ as in Theorem \ref{tm:union}. If $X=\bigcup_{n\in\omega} A_n$ for some $A_n\in \mbb{A}_X$ then $MA_{\pi(X)}(countable)$ implies that $X$ is $\mbb A$-separable.
\end{lemma}
\begin{proof}
 We fix a $\pi$-base $\mathcal{B}$ of $X$ of size $\pi(X)$ and a sequence of dense sets $D_n\subseteq X$. Let us define $D_B=\{p\in \mbb C: \exists n\in \dom(p):D_n\cap A_{p(n)}\cap B\neq \emptyset\}$ for $B\in \mc B$ where $\mbb C$ denotes the Cohen poset; note that each $D_B$ is dense in $\mbb C$. Consider a filter $G\subseteq \mbb C$ which is generic to $\{D_B:B\in \mc B\}$; this exists by $MA_{\pi(X)}(countable)$. Let $g=\cup G$ and define $E_n=D_n\cap A_{g(n)}\in \mbb{A}_X$. 

We claim that $\bigcup_{n\in\omega}E_n$ is dense in $X$. It suffices to show that for every $B\in \mc B$ there is $n\in\omega$ such that $B\cap E_n\neq \emptyset$. As $G\cap D_B\neq\emptyset$, there is $n\in\omega$ such that $A_{g(n)}\cap D_n\cap B\neq \emptyset$; that is $E_n\cap B\neq \emptyset$. 
\end{proof}

An uncountable  space $X$ is {\em Luzin} iff every nowhere dense subset of $X$ is countable. 
We continue by our main theorem:

\begin{example}\label{ex:unified}
It is consistent that there is a left-separated in type $\oo$, first countable,   
0-dimensional Hausdorff space of size $\oo$ such that 
$X$ has a partition $X=Z\cup T\cup Y$ into dense uncountable subspaces such that 
\begin{enumerate}[(1)]
\item $T$ is D-separable;
\item $Y$ is NWD-separable, but not d-separable;
 \item $Z$ is Luzin, so it is not nwd-separable.
\end{enumerate}
Moreover,
\begin{enumerate}[(1)]\addtocounter{enumi}{3}
 \item $T\cup Z$ is d-separable but not  NWD-separable.
\item $Y\cup Z$ is nwd-separable, but not d-separable and not NWD-separable.
\item $T\cup Y$ is d-separable, NWD-separable, but not D-separable.
\end{enumerate}
\end{example}

\begin{figure}[h!]

\begin{tikzpicture}
[pro/.style={rectangle, inner sep=4pt,minimum size=4mm,draw=black},
prosplit/.style=  {inner sep=4pt,minimum size=4mm,draw=black,rectangle split,rectangle split parts=2,
rectangle split part fill={white, black!10}}]

\matrix[row sep=2mm,column sep=1mm]{

&\node (g-1) {};&&\node (g1) {};&\node (g2) {};&\node (g3) {};&\node (g4) {};
&&\node (g5) {};&\node (g6) {};&\node (g7) {};&\node (g8) {};& \node  {$T\cup Y$};\\
&\node (f-1) {};&&\node (f1) {};&\node (f2) {};&\node (f3) {};&\node (f4) {};
&&\node (f5) {};&\node (f6) {};&\node (f7) {};&\node (f8) {};&\node  {$Y$};\\
&&\snode {D-separable}{D-separable}; &&&&& 
\snode{NWD-separable}{NWD-separable}
;\\
\node (e-2) {};&\node (e-1) {};&&\node (e1) {};
&\node (e2) {};&\node (e3) {};&\node (e4) {};
&&\node (e5) {};&\node (e6) {};&\node (e7) {};&\node (e8) {};
\\
&\node (d-1) {};&&\node (d1) {};&\node (d2) {};&\node (d3) {};&\node (d4) {};
&&\node (d5) {};&\node (d6) {};&\node (d7) {};&\node (d8) {};&\node  {$T\cup Z$};
\\
&\node (c-1) {};&&\node (c1) {};&\node (c2) {};&\node (c3) {};&\node (c4) {};
&&\node (c5) {};&\node (c6) {};&\node (c7) {};&\node (c8) {};&\node  {$Y\cup Z$};
\\
&&\snode {d-separable}{d-separable}; 
&&&&& 
\snode {nwd-separable}{nwd-separable};\\
&&\node (b1) {};&\node (b2) {};&\node (b3) {};&\node (b4) {};
&&&\node (b5) {};&\node (b6) {};&\node (b7) {};&\node (b8) {};& \node  {$Z$};\\
\node (a-2) {};&
\node (a-1) {};&&\node (a1) {};&\node (a2) {};&\node (a3) {};&\node (a4) {};
&&\node (a5) {};&\node (a6) {};&\node (a7) {};&\node (a8) {};\\
};

\path[->]

  (D-separable) edge[thick] (NWD-separable) 
  (D-separable) edge[thick] (d-separable)

 (NWD-separable) edge[thick]  (nwd-separable) 
 (d-separable)  edge[thick] (nwd-separable)
;

\path[-, dashed]  (a5.south) edge[thin] (b5.south) (b5.south)  edge[thin] (b7.south) ;

\path[-, dashed]  (a4.south) edge[thin] (c4.south) (c4.south)  edge[thin] (c7.south) ;

\path[-, densely dotted]  (a-1.south) edge[thin] (d-1.south) (d-1.south) edge[thin] (d7.south) ;

\path[-, dashed]  (a-2.south) edge[thin] (e-2.south) (e-2.south) edge[thin] (e2.south)
                   (e2.south) edge[thin] (g2.south)
                   (g2.south) edge[thin] (g7.south);
 ;

\path[-, loosely dashed]  (a3.south) edge[thin] (f3.south) (f3.south) edge[thin] (f7.south) ;

\end{tikzpicture}

 \caption{}

\end{figure}

\begin{proof}[Proof of Theorem \ref{ex:unified}]
First we show that (4)-(6) follows automatically 
from (1)--(3):

\begin{enumerate}[(1)]\addtocounter{enumi}{3}
 \item $T$ is  a dense, D-separable subspace of $T\cup Z$, 
so $T\cup  Z$ is d-separable.   
$Z$ is  a dense, not nwd-separable subspace of $T\cup Z$,
so $T\cup  Z$ is not NWD-separable.

\item $Y$ is  a dense, NWD-separable subspace of $Y\cup Z$, 
so $Y\cup  Z$ is nwd-separable.
$Z$ is  a dense, not nwd-separable subspace of $Y\cup Z$,
so $Y\cup  Z$ is not NWD-separable.

Assume that $\{F_n:n\in {\omega}\}$ are discrete subspaces of 
$Y\cup Z$, and let $F=\bigcup\{F_n:n\in {\omega}\}$. 
Then $F_n\cap Z$ is discrete, hence nowhere dense and so
countable. Thus  $F\cap Z$ is countable.
Since $X$ is left-separated in type $\oo$, it follows that
$F\cap Z$ is nowhere dense.
But $Y$ is not d-separable, so $X\ne \overline{F\cap Y}$.
Thus $X \ne \overline{F}$, i.e. $Y\cup Z$ is not d-separable.

\item $T\cup Y$ is the union of two NWD-separable spaces, so it is NWD-separable 
by Theorem \ref{cor:union}.
$T$ is  a dense, D-separable subspace of $T\cup Y$, 
so $T\cup  Y$ is d-separable.
$Y$ is  a dense, not d-separable subspace of $Y\cup Z$, 
so $Y\cup  Z$ is not D-separable.
\end{enumerate}

Now we define a poset $\mbb Q$ which   has property K, thus ccc, and forces a
left-separated, first countable,  0-dimensional Hausdorff   
topology on the set $X = \oo \times ({\omega}+1)$ such that $X$ has a 
partition $X=Z\cup T\cup Y$ into dense uncountable subspaces with the 
following properties: 
\begin{enumerate}[(A)]
\item $T$ is $\sigma$-discrete,
\item $Y$ is $\sigma$-nowhere dense and $s(Y)={\omega}$, i.e. every discrete subset of $Y$ is countable,
 \item $Z$ is Luzin.
\end{enumerate}  
We will do this in
such way, that a condition $p\in  \mbb Q$ will be a finite approximation of a countable
neighborhood base.

For $i<2$ let $I_i=\{n\in {\omega}: n \equiv i \mod 2\}$, and let 
$I_2=\{{\omega}\}$.
The underlying set of $T$, $Y$ and $Z$ will be 
 $\oo\times I_0$, $\oo\times I_1$ and $\oo\times I_2$, respectively.

We will use the following notations: if $x = ({\alpha}, k) \in X$ let
\begin{equation}
Q_x=\left \{
\begin{array}{ll}
\text{$[{\alpha},{\omega}_1)\times [k+1,{\omega}]\cup \{x\}$}&\text{if $k\in I_0,$}\\
\text{$[{\alpha},{\omega}_1)\times [k,{\omega}]$}&\text{if $k\in I_1,$}\\
\text{$[{\alpha},{\omega}_1)\times [0,{\omega}]$}&\text{if $k\in I_2,$}
\end{array}
\right . 
\end{equation}
see Figure 4.

\begin{figure}[h!]

\begin{tikzpicture}[scale=0.8]

\draw    (0,1)  node {$0$};
\draw[thin]   (1,1)  -- (10,1);

\draw    (0,2)  node {$1$};
\draw[thin]   (1,2)  -- (10,2);
\node (x) at (3,2)  {$\bullet$};
\node[below]  at (x.east)  {$x$};
\node (xx) at (2.5,1.5)  {};
\node[left]  at (xx.west)  {$Q_x$};
\draw[thick]   (2.5, 4.5) --   (2.5, 1.5)   --   (6,1.5);

\draw    (0,3)  node {$2$};
\draw[thin]   (1,3)  -- (10,3);
\node (y) at (6,3)  {$\bullet$};
\node[below]  at (y.east)  {$y$};
\node (yy) at (5.5,2.5)  {};
\node[left]  at (yy.west)  {$Q_y$};
\draw[thick, dashed]    (5.5, 5.5) --    (5.5, 2.5)  --  (6.5, 2.5)  --  (6.5, 3.5)   --   (9.5, 3.5)    ;

\draw[thin]   (1,4)  -- (10,4);

\draw    (0,6)  node {$\omega$};
\draw[thin]   (1,6)  -- (10,6);
\node (z) at (4,6)  {$\bullet$};
\node[below]  at (z.east)  {$z$};
\node (zz) at (3.5,6.5)  {};
\node[left]  at (zz.west)  {$Q_z$};
\draw[thick]   (3.5, 4.5) --   (3.5, 6.5)   --   (8.5,6.5);

\end{tikzpicture}
 \caption{}
\end{figure}

Let $\mbb Q$ consist of the following conditions
\begin{equation}\notag
p=\<I^p, n^p, \<U^p(x,j):x\in I^p, j<n^p\>\> 
\end{equation}
such that 
\begin{enumerate}[({Q}-a)]
 \item  $I^p\in \br \oo\times ({\omega}+1);<{\omega};$  and $n^p\in {\omega}$,
\item  $x\in U^p(x, j) \subs I^p\cap  Q_x$ for all $x \in  I^p$ and  $j < n^p$.
\end{enumerate}
If $ p, q \in  \mbb Q$ let $q \le p$ iff
\begin{enumerate}[({Q}-i)]
 \item $I^p \subseteq I^q$, $n^p \le n^q$ and $U^p(x, j) \subseteq U^q(x, j)$ for all 
$x \in  I^p$ and $j < n^p,$
\item $U^p(x, j) \subseteq U^p(y, k)$ implies $U^q(x, j) \subseteq U^q(y, k)$ for all 
$x, y \in  I^p$ and $j, k < n^p,$
\item $U^p(x, j) \cap U^p(y, k) = \empt$ implies $U^q(x, j) \cap U^q(y, k) = \empt$ 
for all $x, y \in  I^p$ and
$j, k < n^p.$
\end{enumerate}

If $G$ is a generic filter in $\mbb Q$ then let
\begin{equation}
U^G(x, j) = \{U^p(x, j) : p \in G, x \in I^p, j < n^p\} 
\end{equation}
for any $x \in  X$ and $j < \omega$. Let 
$B^G(x) = \{U^G(x, j) : j < {\omega}\}$ for $x \in X$.

\begin{lemma}
$\bigcup \{B^G(x) : x \in  X\}$ forms a base for a Hausdorff, 
0-dimensional topology $\tau^G$ on $X$, such that $B^G(x)$ is a countable neighborhood base
for the point $x \in  X$. $X$ is left separated in type $\oo$.
\end{lemma}

\begin{proof}
 The statement follows from standard density arguments; for details we refer to \cite{juhi}.
\end{proof}

Let $E_n=\oo\times \{n\}$ for $n\in {\omega}$. The next lemma follows from easy density arguments as well:

\begin{lemma}
\begin{enumerate}[(a)]
 \item  The subspace $E_n$ is discrete in $X$ for $n\in I_0$; hence $T$ is $\sigma$-discrete. 
\item The subspace $E_n$ is nowhere dense  in $Y$ for $n\in I_1$; hence $Y$ is $\sigma$-nowhere dense.
\item $\oo\times \{{\omega}\}$ is dense in $X$.
\item If $N\subs {\omega}$ is infinite then $\oo\times N$ is dense in $X$.
\end{enumerate}
 
\end{lemma}

Denote $\pi$ the projection from $\oo\times ({\omega}+1)$ onto $\oo$, i.e.
$\pi(\<{\alpha},n\>)={\alpha}$.

\begin{definition}
 We say that the conditions $p$ and $q$
are {\em twins} iff  $n^p=n^q$, $|\pi[I^p]|=|\pi[I^q]|$ and
denoting by ${\sigma}$ the unique $<$-preserving bijection between $\pi[I^p]$ and
$\pi[I^q]$ we have
\begin{enumerate}
 \item $\pi[I^p]\cap\pi[I^q]<\pi[I^p]\setm \pi[I^q]<\pi[I^q]\setm\pi[I^p]$,
\item 
using the notation $\ssigma(\<{\alpha},n\>)=\<\sigma({\alpha}),n\>$,
\begin{enumerate}[(i)]
\item $I^q=\ssigma[I^p]$,
\item $U^q(\ssigma(x),i)=\ssigma[U^p(x,i)]$ for $x\in I^p$ and $i<n^p$.
\end{enumerate}
\end{enumerate}
We say that $\ssigma$ is the {\em twin function} from $p$ to $q$.
\end{definition}

The following lemma is rather technical although it will be essential in finishing our proof. We encourage the reader to skip the proof of Lemma \ref{lm:main_uniform} at first read and see its applications in what follows. 

\begin{lemma}\label{lm:main_uniform}
Assume that $p$ and $q$ are twin conditions, $n^p=n^q=n$, and
$\ssigma$ is the twin function from $p$ to $q$.

\noindent (i) If $u\in I^p\cap (\oo\times I_1)$, then there is 
a condition $r\le p,q$ such that 
\begin{equation}
  \ssigma(u)\in \bigcap_{\ell<n^p}U^r(u,\ell).
\end{equation}
(ii) If
\begin{enumerate}[(a)] 
\item $r\le p$, $I^r=I^p,$ $n^r=n+1$, $U^r(x,i)=U^p(x,i)$ for $i<n$ and  $U^r(x,n)=\{x\}$ for each $x\in I^r$.
\item $s\le r$ such that $\pi[I^s] < \pi[I^q\setm I^p]$,
\item  $u\in (I^p\setm I^q)\cap (\oo\times I_2)$, and   $v\in U^s(u,n)$, 
\end{enumerate}
then there is a condition $t\le s,q$ such that 
\begin{equation}
 \ssigma(u)\in \bigcap_{i<n^s} U^t(v,i).
\end{equation}
 
\end{lemma}

\begin{proof}
$(i)$
Define $r=\<I^r, n^r, \<U^r(x,j):x\in I^r, j<n^r\>\>$ as follows:
\begin{itemize}
 \item $I^r=I^p\cup I^q$, $n^r=n^p=n^q$.
\item if $x=({\alpha},n)\in I^r$, let
\begin{equation}\notag
U^r(x,j)=\left \{
\begin{array}{ll}
U^p(x,j)\cup U^q(x,j)&\text{if $x\in I^p\cap I^q,$}\\ 
U^p(x,j)\cup \{\ssigma(u)\}&\text{if $x\in I^p\setm I^q$  and $u\in U^p(x,j)$, }\\ 
U^p(x,j)&
\text{if $x\in I^p\setm I^q$ and  $u\notin U^p(x,j)$, }\\ 
U^q(x,j)&\text{if $x\in I^q\setm I^p$.}\\ 
\end{array}
\right . 
\end{equation}
\end{itemize}
It is clear that $r$ satisfies (Q-a). If   (Q-b) fails then
$U^r(x,j)\not\subs Q_x$ for some $x\in I^r$.
Since $p$ and $q$ are twins, the only possibility is that 
$U^r(x,j)=U^p(x,j)\cup \{\ssigma(u)\}$ and $\ssigma(u)\notin Q_x$.
But $u\in Q_x$, so the only possibility is that $x=u$.
But $u\in \oo\times I_1$, so $\ssigma(u)\in Q_u$.
So $r$ satisfies (Q-b) as well. Hence 
 $r\in \mbb Q$.
 
To show $q\le p,q$, first remark that (Q-i) is trivial by the construction.

(Q-ii) can be easily seen to hold since $U^p(x,j)\subs U^p(y,k)$ iff $U^q(\ssigma(x),j)\subs U^q(\ssigma(y),k)$
as $p$ and $q$ are twins.

To check (Q-iii) assume first that $x,y\in I^p$  and $U^p(x, j) \cap U^p(y, k)=\empt $.
We can assume that $u\notin U^p(x,j)$ and so $\ssigma(u)\notin U^r(x,j)$.
Thus
\begin{multline}
U^r(x,j)\cap U^r(y,k)\subs\\  \bigl(U^p(x,j)\cup U^q(\ssigma(x), j)\bigr)\cap
\bigl(U^p(y,k)\cup U^q(\ssigma(y), k)\bigr).
\end{multline}
But $p$ and $q$ are twins, and so $U^p(x, j) \cap U^p(y, k)=\empt $ implies 
\begin{equation}
\bigl(U^p(x,j)\cup U^q(\ssigma(x), j)\bigr)\cap
\bigl(U^p(y,k)\cup U^q(\ssigma(y), k)\bigr)=\empt
\end{equation}
as well. 

Thus $r$ and $p$ satisfy (Q-iii), and so $r\le p$.

Now let  $x,y\in U^q$  such that  $U^q(x, j) \cap U^q(y, k)=\empt $.

Pick $x',y'\in I^p$ with $\ssigma(x')=x$ and
$\ssigma(y')=y$.

Assume that $\ssigma(u)\notin U^q(x,j)$.
Then $\ssigma(u)\notin U^r(x,j)$.
Thus
\begin{multline}
U^r(x,j)\cap U^r(y,k)\subs\\  \bigl(U^q(x,j)\cup U^p(x', j)\bigr)\cap
\bigl(U^q(y,k)\cup U^p(y', k)\bigr).
\end{multline}
But $p$ and $q$ are twins, and so $U^q(x, j) \cap U^q(y, k)=\empt $ implies 
\begin{equation}
U^r(x,j)\cap U^r(y,j)\subs\\  \bigl(U^q(x,j)\cup U^p(x', j)\bigr)\cap
\bigl(U^q(y,k)\cup U^p(y', k)\bigr)
=\empt
\end{equation}
as well. 

Thus $r$ and $q$ satisfy (Q-iii), and so $r\le q$.

So we proved $r\le p,q$.

Finally  
$\ssigma(u)\in U^r(u,\ell)$ for $\ell <n^p$  is clear from the construction.
This proves \ref{lm:main_uniform}$(i)$.

\bigskip

\noindent$(ii)$\\
Define $t=\<I^t, n^t, \<U^t(x,j):x\in I^t, j<n^t\>\>$ as follows.
For $x
\in I^s$ and $j<n^s$
let 
\begin{equation}\label{eq:v}
V(x,j)=\bigcup\{U^q(z,\ell):z\in I^p\cap I^q, \ell<n^p, U^s(z,\ell)\subs U^s(x,j)\}, 
\end{equation}
and 
\begin{equation}\label{eq:w}
W(x,j)=\left \{
\begin{array}{ll}
\{\ssigma(u)\}&\text{if $U^s(v,i)\subs U^s(x,j)$ for some $i<n^s$,}\\
\empt&\text{otherwise.}
\end{array}
\right . 
\end{equation}
Let
\begin{enumerate}[(a)]
 \item $I^t=I^s\cup I^q$, $n^t=n^s$.
\item For $x\in I^t$ and $j<n^t$ let
\begin{equation}\label{eq:ut}
U^t(x,j)=\left \{
\begin{array}{ll}
U^s(x,j)\cup V(x,j)\cup W(x,j)&\text{if $x\in I^s,$}\\ 
U^q(x,j)&\text{if $x\in I^q\setm I^s$ and $j<n^q$,}\\ 
\{x\}&\text{if $x\in I^q\setm I^s$ and $n^q\le j< n^t$.}
\end{array}
\right . 
\end{equation}
\end{enumerate}
Clearly $t$ satisfies (Q-a). 

Assume on the contrary  that $w\in U^t(x,j)\setm Q_x$ witnesses that 
 (Q-b) fails.
Since $\ssigma(u)\in E_{\omega}$, we have $w\ne \ssigma(u)$.

So $w\in V(x,j)\setm U^s(x,j)\subs V(x,j)\setm I^p$.
Pick $w'\in I^p\setm I^q$ with $\ssigma(w')=w$.
There is $z\in I^p\cap I^q$ and $\ell<n^p$
such that $w\in U^q(z,\ell)$ and $U^s(z,\ell)\subs U^s(x,j)$.
Thus $x\notin I^p\setm I^q$, and so $x\ne w'$.
Thus  $w'\in U^p(z,\ell)\subseteq U^s(z,\ell)\subseteq U^s(x,j)\subs Q_x$ implies $w=\ssigma(w')\in Q_x$.
This contradiction shows that (Q-b) must hold. Thus, we proved that $t\in \mbb Q$.

Check $t\le s,q$. 
(Q-i) is trivial. 

(Q-ii) holds for $t\le s$, because the construction     is "monotone" in 
(\ref{eq:v})--(\ref{eq:ut}).
(Q-ii) also holds for $t\le q$ because if $U^q(x,j)\subs U^q(y,k)$ then it 
is not possible
that $x\in I^p\cap I^q$ and $y\in I^q\setm I^p$, so we can use that the 
construction     is "monotone". 

Now check (Q-iii) for $t\le s$. 
So let $U^s(x,j)\cap U^s(y,k)=\empt$, and assume on the contrary that 
$U^t(x,j)\cap U^t(y,k)\ne\empt$.
Since $p$ and $q$ are twins, we have 
\begin{equation}\label{eq:disjoint}
(U^s(x,j)\cup V(x,j))\cap
(U^s(y,k)\cup V(y,k))=\empt. 
\end{equation}

Indeed, assume that $w\in (U^s(x,j)\cup V(x,j))\cap
(U^s(y,k)\cup V(y,k))$. Since $V(x,j)\cap I^s\subs U^s(x,j)$, we can assume 
$w\in I^q\setm I^p,$ i.e. $w\in V(x,j)\cap V(y,k)$.
Then ${\ssigma}^{-1}(w)\in U^s(x,j)\cap U^s(y,k)$. Contradiction.

\medskip

So, by (\ref{eq:disjoint}), $U^t(x,j)\cap U^t(y,k)\ne \empt$ implies 
$\ssigma(u)\in U^t(x,j)\cap U^t(y,k)$.
Since $u,v\notin U^s(x,j)\cap U^s(y,k)$, we can assume that 
\begin{equation}
 \ssigma(u) \in W(x,j)\quad \text{ and } \quad \ssigma(u)\in V(y,k),  
\end{equation}
so
\begin{equation}
 U^s(v,i)\subs U^s(x,j) \text{ and } 
u\in U^s(z,\ell)\subs U^s(y,k),  
\end{equation}
for some $i<n^s$, $ z\in I^p\cap I^q$ and $\ell<n^p$.

So $U^s(z,\ell)\cap U^s(u,n)\ne\empt$. Thus $s\le r$ implies
$U^r(z,\ell)\cap U^r(u,n)\ne\empt$, that is $u\in U^r(z,\ell)$. But $U^r(u,n)=\{u\}$,
so $U^r(u,n)\subs U^r(z,\ell)$. Thus $U^s(u,n)\subs U^s(z,\ell)$,
and so $v\in U^s(u,n)\subs U^s(z,\ell)\subs U^s(y,k)$.
Thus $v\in U^s(x,j)\cap U^s(y,k)$. Contradiction, thus
$U^t(x,j)\cap U^t(y,k)=\empt$.

Finally check (Q-iii) for $t\le q$. 
So let $U^q(x,j)\cap U^q(y,k)=\empt$.

We should distinguish three cases as follows.

If $x,y\in I^p\cap I^q$, then $x,y\in I^s$. As $U^p(x,j)\cap U^p(y,k)=\empt$ and $s\leq p$, we have that $U^s(x,j)\cap U^s(y,k)=\empt$. We have just verified that  (Q-iii) holds in this case.

If $x, y\in I^q\setm I^p$, then $U^t(x,j)=U^q(x,j)$ and $U^t(y,k)=U^q(y,k)$, so (Q-iii) is trivial.

Finally let $x\in I^p\cap I^q$ and $y\in I^q\setm I^p$.
Then 
\begin{equation}
U^t(x,j)\cap U^t(y,k)=\bigl (U^s(x,j)\cup V(x,j)\cup W(x,j)\bigr )\cap U^q(y,k)
\end{equation}
Let $y'={\ssigma}^{-1}(y)\in I^p\setm I^q$.
Then $U^p(x,j)\cap U^p(y',k)=\empt$, and so 
$U^s(x,j)\cap U^s(y',k)=\empt$.
Since $U^q(y,k)\cap I^s\subs  U^p(y',k)\subs I^s(y',k)$, we have 
\begin{equation}
U^s(x,j)\cap U^q(y,k)=\empt.
\end{equation}

If $U^s(z,\ell)\subs U^s(x,j)$, then
$U^s(x,j)\cap U^s(y',k)=\empt$ implies that  
$U^s(z,\ell)\cap U^s(y',k)=\empt$, and so
$U^p(z,\ell)\cap U^p(y',k)=\empt$. Thus  
$U^q(z,\ell)\cap U^q(y,k)=\empt$. Hence 
\begin{equation}
V(x,j)\cap U^q(y,k)=\empt. 
\end{equation}

Assume that $W(x,j)=\{\ssigma(u)\}$.
Then $
 U^s(v,i)\subs U^s(x,j)$ for some $i<n^s$.
Thus $U^s(u,n)\cap U^s(x,j)\ne \empt$, so 
$U^r(u,n)\cap U^r(x,j) \ne\empt$, so $u\in U^p(x,j)$.
Thus $u\notin U^p(y'k)$ and so 
$\ssigma(u)\notin U^q(y,k)$. 

So $U^t(x,j)\cap U^t(y,k)=\empt.$

Thus $t\leq s,q$ and $\ssigma(u)\in \bigcap_{i<n^s} U^t(v,i)$ holds as well. 

\end{proof}

\begin{lemma}
$\mbb Q$ has property K.  
\end{lemma}

\begin{proof}
If $\<p_{\alpha}:{\alpha}<\oo\>\subs \mbb Q$, then 
by standard $\Delta$-system arguments we can find an uncountable $I\subs \oo$
such that $p_{\alpha}$ and $p_{\beta}$ are twins whenever ${\alpha}<{\beta}\in I$.
So $p_{\alpha}$ and $p_{\beta}$ are compatible by  Lemma \ref{lm:main_uniform}.
\end{proof}

\begin{lemma}\label{lm:univ_en_small_discrete}
If $m\in I_1$ then $E_m$ does not contain any uncountable discrete subspace; in particular, $s(Y)=\omega$.
\end{lemma}

\begin{proof}
Assume that $p\Vdash$ "$  \dot A=\{\dot x_{\zeta}:{\zeta}<\oo\}\in \br E_m;\oo; $ is discrete".
For each ${\zeta}<\oo$ pick a condition $p_{\zeta}$ which 
decides the value of $\dot x_{\zeta}$
and 
\begin{equation}
 p_{\zeta}\Vdash U(x_{\zeta},k_{\zeta})\cap \dot A=\{x_{\zeta}\}.
\end{equation}
We can assume that  the elements $\{x_{\zeta}:{\zeta}<\oo\}$ are pairwise different, 
 $x_{\zeta}\in I^{p_{\zeta}}$ and  $k_{\zeta}=k<n^{p_{\zeta}}$.

By standard $\Delta$-system arguments we can find
${\zeta}<{\xi}<\oo$ such that $p_{\zeta}$ and $p_{\xi}$ are twins,
and $\ssigma(x_{\zeta})=x_{\xi}$, where $\ssigma$ is a the twin function.

Then, by Lemma \ref{lm:main_uniform} part $(i)$, there is a $q\le p_{\zeta},p_{\xi} $
such that $\ssigma (x_\zeta)=x_{\xi}\in \bigcap_{\ell<n^{p_\zeta}} U^q(x_{\zeta},\ell)$ and so
\begin{equation}
q\Vdash \{ x_{\zeta},x_{\xi}\}\subs U^G(x_{\zeta},k)\cap \dot A. 
\end{equation}
This contradicts the choice of the neighborhoods which finishes the proof.
\end{proof}

\begin{lemma}\label{lm:luzin_i2}
Every uncountable  subset $A$ of $Z$ is somewhere dense in $X$. In particular,  $Z$ is a Luzin subspace of $X$. 
\end{lemma}

\begin{proof}[Proof of Lemma \ref{lm:luzin_i2}]
Assume $p\Vdash \dot A=\{\dot a_{\zeta}:{\zeta}<\oo\}\subs E_{\omega}$.

Pick conditions $\{p_{\zeta}:{\zeta}<\oo\}$,  
and ordinals $\{{\alpha}_{\zeta}:{\zeta}<\oo\}\subs \oo$
such that $p_{\zeta}\Vdash \dot a_{\zeta}=\<{\alpha}_{\zeta},k\>$.
We can assume that 
\begin{enumerate}[(i)]
 \item if ${\zeta}<{\xi}<\oo$ then 
$p_{\zeta}$ and $p_{\xi}$ are twins,  so $n^{p_{\zeta}}=n$.
  \item $a_{\zeta}\in I^{p_{\zeta}}\setm I^{p_{\xi}}$ for ${\xi}\ne {\zeta}$.
\item $\ssigma_{{\zeta},{\xi}}(a_{\zeta})=a_{\xi}$, where
$\ssigma_{{\zeta},{\xi}}$ is the twin function from 
$p_{\zeta}$ to $p_{\xi}$.
\end{enumerate}

Let $r\le p_0$, $I^r=I^{p_0},$ $n^r=n+1$, 
$U^r(x,i)=U^{p_0}(x,i)$ for $i<n$ and 
$U^r(x,n)=\{x\}$ for $x\in I^r$.

\begin{claim}
$r\Vdash$ ``$ A\cap \{a_{\zeta}:{\zeta}<\oo\}$'' is dense in $U^G(a_0,n)$``. 
\end{claim}

Indeed, assume that $s\le r$ such that $s\Vdash v\in U^G(a_0,n)$, i.e.
$v\in U^s(a_0,n)$.  Pick ${\xi}<\oo$ such that 
$\pi[I^s]<\pi[I^{p_{\xi}}\setm I^{p_0}]$.

Then, by lemma \ref{lm:main_uniform}, there is a condition
$t\le s,p_{\xi}$ such that 
\begin{equation}
 a_{\xi}\in \bigcap_{i<n^s} U^t(v,i).
\end{equation}
Thus
\begin{equation}
t\Vdash  \dot A\cap\bigcap_{i<n^s} U^G(v,i)\ne\empt.
\end{equation}
Since $s$ and $v$ were arbitrary, we proved the claim, and so does the lemma.
\end{proof}

Now let $\mbb P = \mbb Q \times \mbb C_{{\omega}_2}$ , where $\mbb C_{{\omega}_2}$ 
is the standard poset adding  ${\omega}_2$ many Cohen-reals.

 Let $G = G_0\times G_1$ be a generic filter in $\mbb P$, such that $G_0$ is generic
in $\mbb Q$. Consider the space $X^{G_0}=(\oo\times ({\omega}+1), \tau^{G_0} )$.
Since $V[G]=V[G_1][G_0]$, it follows that $X^{G_0}$ and the corresponding $T$,
$Y$ and $Z$ satisfy (A)--(C).
However $V[G]=V[G_0][G_1]$ as well, so 
MA${}_{\oo}(countable)$ also holds.

We claim that $T$ is $D$-separable; indeed, $T$ is $\sigma$-discrete, $w(T)=\oo$ and 
MA${}_{\oo}(countable)$ holds hence Lemma \ref{martin} implies that
\begin{enumerate}[(A)]\addtocounter{enumi}{3}
 \item $T$ is D-separable.
\end{enumerate}
Similarly, since $Y$ is $\sigma$-nowhere dense, $w(T)=\oo$ and 
MA${}_{\oo}(countable)$ holds, Lemma \ref{martin} implies that
\begin{enumerate}[(A)]\addtocounter{enumi}{4}
 \item $Y$ is NWD-separable.
\end{enumerate}
Finally observe that (A)--(E) imply (1)--(3). This finishes the proof of the theorem.
\end{proof}

\section{Monotonically normal spaces - positive results} \label{mn}


Barman and Dow's aforementioned result suggests that convergence properties have some influence on selective versions of separability. In this section our first aim is to prove that $nwd$-separability and $D$-separability are equivalent in the class of monotonically normal spaces. This result exploits a weak convergence property which is satisfied by all monotonically normal spaces.


 
\begin{definition}
(\cite{DTTW}) A space $X$ is called \emph{discretely generated} (\emph{nowhere densely generated}) if for every set $A \subset X$ and every point $x \in \overline{A}$ there is a discrete (nowhere dense) $D \subset A$ such that $x \in \overline{D}$.
\end{definition}

The property of being discretely generated (nowhere densely generated) is called \emph{discrete tightness} (\emph{nowhere dense tightness}) by Bella and Malykhin in \cite{BM}. Of course every crowded discretely generated space is nowhere densely generated, but the converse doesn't hold, as the following example shows.

\begin{example}
There is a nowhere densely generated space which is not discretely generated.
\end{example}

\begin{proof}

Let $X$ be the set of all countably supported functions in $2^{\omega_1}$ with the countably supported box product topology, and $Y$ be any countable non-discretely generated space (for example, a countable maximal space). We claim that $X \times Y$ is the desired example.

\vspace{.1in}

\noindent {\bf Claim 1.} Every meager set is nowhere dense in $X$.

\begin{proof}[Proof of Claim 1]
Let $\{N_n: n < \omega \}$ be a countable family of nowhere dense sets in $X$, $\sigma \in Fn(\omega_1, 2, \omega_1)$ and define $[\sigma]:=\{f \in X: f \supset \sigma \}$. Since $N_0$ is nowhere dense, the set $[\sigma] \setminus \overline{N_0}$ is non-empty, and thus we can find a countable partial function $\sigma_0 \supset \sigma$ such that $[\sigma_0] \subset [\sigma] \setminus \overline{N_0}$. Suppose we have constructed an increasing sequence of partial functions $\{\sigma_k: k \leq n \}$. Since $N_{n+1}$ is nowhere dense we can find a partial function $\sigma_{n+1} \supset \sigma_n$ such that $[\sigma_{n+1}] \subset [\sigma] \setminus \overline{N_{n+1}}$. Let $\sigma_{\omega}=\bigcup_{n<\omega} \sigma_n$, which is a countable partial function since $\{\sigma_n: n <\omega \}$ is a sequence of compatible countable partial functions. Then $[\sigma_\omega]$ is a non-empty open set contained in $[\sigma]$ and disjoint from $\overline{\bigcup_{n<\omega} N_n}$. This shows that $[\sigma]$ cannot be contained in $\overline{\bigcup_{n<\omega} N_n}$ and thus this latter set is nowhere dense.

\renewcommand{\qedsymbol}{$\triangle$}
\end{proof}

 \noindent {\bf Claim 2.} The space $X$ is discretely generated.
 
 \begin{proof}[Proof of Claim 2]
 Note that the character of a point $x \in X$ is equal to $\aleph_1$. Let $A \subset X$ be a non-closed set and $x \in \overline{A} \setminus A$. Since $X$ is a $P$-space we can fix a decreasing local base $\{U_\alpha: \alpha < \aleph_1 \}$ at $x$. For every $\alpha < \aleph_1$ pick $x_\alpha \in U_\alpha \cap A$. Then $S=\{x_\alpha: \alpha < \aleph_1\}$ converges to $x$. If $S$ had another accumulation point $y \neq x$, then, since $X$ is a $P$-space, every neighbourhood of $y$ should hit $S$ into uncountably many points. But that contradicts convergence. So $S$ is a discrete set such that $x \in \overline{S}$.
 \renewcommand{\qedsymbol}{$\triangle$}
  \end{proof}
 
 The space $X \times Y$ is not discretely generated because it contains a homeomorphic copy of $Y$. Let $A \subset X \times Y$ and $(x,y) \in \overline{A}$. 
 
 Let $\{y_n: n< \omega \}$ be an enumeration of the set $\pi_Y(A)$ and set $P_n=\{z \in X: (z, y_n) \in A \}$. Moreover define $B \subset Y$ to be the set $$B= \{y_n: x \in \overline{P_n} \}.$$
 
 \noindent {\bf Claim 3.} The point $y$ is in the closure of $B$.
 
 \begin{proof}[Proof of Claim 3] Suppose that this is not the case and let $V$ be a neighbourhood of $y$ which misses $B$. Let $S \subset \omega$ be the set such that $V \cap \pi_Y(A)= \{y_n: n \in S \}$. For every $n \in S$ we have that $x \notin \overline{P_n}$, and thus we can find an open neighbourhood $U_n$ of $x$ such that $U_n \cap P_n=\emptyset$. But then $(\bigcap_{n<\omega} U_n) \times V$ is a neighbourhood of $(x,y)$ which misses $A$ and this is a contradiction
 \renewcommand{\qedsymbol}{$\triangle$}
 \end{proof}
 
 Let $T \subset \omega$ such that $B=\{y_n: n \in T \}$. For every $n \in T$ we have that $x \in \overline{A_n}$, so, by Claim 2, there is a discrete $D_n \subset A_n$ such that $x \in \overline{D_n}$. Now since $X$ is dense-in-itself and $D_n$ is nowhere dense, by Claim 2 we have that $\bigcup_{n \in T} D_n$ is nowhere dense. Thus the set $N:=\bigcup_{n \in T} D_n \times \{y_n\} \subset A$ is also nowhere dense and it is easy to see that $(x,y) \in \overline{N}$. This proves that $X \times Y$ is nowhere densely generated.
\end{proof}

These convergence-type properties are very useful in our context and this is apparent from the following fact.


\begin{fact} \label{easyfact}
Every separable discretely generated (nowhere densely generated) space is $D$-separable ($NWD$-separable).
\end{fact} 

Of course, we would be happier to obtain a relationship between $d$-separability and $D$-separability, but unfortunately, we already saw that there can be even first-countable, $d$-separable spaces which are not $D$-separable, so there is no way to simply replace separability with $d$-separability in Fact $\ref{easyfact}$. Another approach would be to try and strengthen discrete generability to something more suitable to our purposes. This amounts to nothing more than replacing points with discrete sets:

\begin{definition} \label{ddg}
A space is \emph{discretely discretely generated} (in short, DDG) if for every set $A \subset X$ and every discrete set $D \subset \overline{A}$ there is a discrete set $E \subset A$ such that $D \subset \overline{E}$.
\end{definition}

\begin{fact}
\cite{AJR} Every discretely discretely generated $d$-separable space is $D$-separable.
\end{fact}

%


The authors of \cite{AJR} proved that every monotone normal space is DDG; hence monotone normal, $d$-separable spaces are $D$-separable. We need the following closely related result:

\begin{lemma} \label{MNlemma}
Let $X$ be a monotonically normal space, $A \subset X$ be a dense set and $N \subset X$ be nowhere dense. Then there is a discrete set $D \subset A$ such that $N \subset \overline{D}$.
\end{lemma}

\begin{proof}[First Proof]
Let $\mc U$ be a maximal system of pairs $\<x,U\>\in A\times \tau_X$
such that 
\begin{enumerate}[(1)]
 \item $x\in U\subs X\setm N$,
\item $\<x,U\>\ne \<x',U'\>\in \mc U$ then $H^2(x,U)\cap H^2(x',U')=\empt$.
\end{enumerate}

Let $D=\{x:\<x,U\>\in \mc U\}$. Clearly $D\subs A$ is discrete.

We show $N\subs \overline{D}$.
Let $y\in N$. Assume on the contrary that $y\in W\in \tau_X$ with $W\cap D=\empt.$

If $\<x,U\>\in \mc U$, then $y\notin U\subs X\setm N$ and $x\notin W$
so 
\begin{equation}
 H(y,W)\cap H(x,U)=\empt.
\end{equation}
Let $V\subs H(y,W)\subs X\setm N$ be open and pick $z\in A\cap V$.
Then $z\notin H(x,U)$ and $x\notin V$.
Thus
\begin{equation}
 H(z,V)\cap H^2(x,U)=\empt.
\end{equation}
Thus $\mc U$ was not  not maximal because $\mc U\cup \{\<z,V\>\}$ also satisfies 
(1) and (2).
Contradiction. 
\end{proof}

\begin{proof}[Second Proof]
Suppose you constructed open sets $\{U_\alpha: \alpha < \beta \}$, points $\{x_\alpha: \alpha < \beta \} \subset A$ such that:

\begin{enumerate}
\item $x_\alpha \in U_\alpha$.
\item $U_\alpha \cap \overline{N}=\emptyset$.
\item $\{H(x_\alpha, U_\alpha): \alpha < \beta \}$ is a pairwise disjoint family.
\end{enumerate}

If $N \nsubseteq \overline{\bigcup \{H(x_\alpha, U_\alpha): \alpha < \beta \}}$ then use the fact that $A$ is dense to choose $x_\beta \in A \setminus  \overline{\bigcup \{H(x_\alpha, U_\alpha): \alpha < \beta \}}$ such that $x_\beta \notin \overline{N}$. Now choose a neighborhood $U_\alpha$ of $x_\alpha$ such that $U_\alpha \cap \overline{N}=\emptyset$.

Let $\gamma$ be the least ordinal such that $N \subset \overline{\bigcup \{H(x_\alpha, U_\alpha): \alpha < \beta \}}$. We claim that $D=\{x_\alpha: \alpha < \gamma \}$ is the required discrete set. 
Indeed, suppose by contradiction that there is $y \in N \setminus \overline{\{x_\alpha: \alpha < \gamma\}}$. We have $H(y, X \setminus \overline{\{x_\alpha: \alpha < \beta \}}) \cap H(x_\tau, U_\tau) \neq \emptyset$ for some $\tau < \beta$. 
So either $x_\tau \in X \setminus \overline{\{x_\alpha: \alpha < \gamma\}}$ or $y \in U_\tau$, but both lead to a contradiction.
\end{proof}

\begin{theorem} \label{MNthm}
Every monotonically normal, $nwd$-separable space is $D$-separable.
\end{theorem}


\begin{proof}
Let $X$ be a monotonically normal space with a $\sigma$-nowhere dense set $D=\bigcup_{n<\omega} N_n$. Fix a sequence of dense sets $\{D_n:n\in\omega\}$ as well. Let us apply Lemma $\ref{MNlemma}$ to pick discrete sets $E_n \subset D_n$ such that $N_n \subset \overline{E_n}$. Then $\bigcup_{n<\omega} E_n$ is a dense subset of $X$ and this witnesses that $X$ is $D$-separable.
\end{proof}

The following theorem can be derived from Theorem $\ref{MNthm}$ and Theorem 28 of \cite{BMS}. We offer an alternative proof based on Mary Ellen Rudin's famous result that every compact monotonically normal space is the continuous image of a compact linearly ordered space.

\begin{theorem}
Every compact, monotonically normal, $nwd$-separable space has a $\sigma$-disjoint $\pi$-base.
\end{theorem}

\begin{proof}
Let us remark that $X$ does not have isolated points because it is
nwd-separable.

Assume first that $X$ is a GO-space, i.e., it is a  subspace of an ordered space $Y$.

Let $\{N_n:n\in {\omega}\}$ be a family of nowhere dense subsets of $X$ such that
$\cup\{N_n:n\in {\omega}\}$ is dense.  We can assume $N_0\subs N_1\subs \dots$.

For each $n\in {\omega}$ consider $Y\setm \overline{N_n}$, and let 
$\mc U_n$ be the natural partition of $Y\setm \overline{N_n}$ into maximal convex sets.
Let $\mc V_n=\{U\cap X: U\in \mc U_n\}$.

We claim that $\mc V=\bigcup_{n\in {\omega}}\mc V_n$ is a $\pi$-base.

Indeed, let $(y,y')$ be an open interval with $X\cap (y,y')\ne \empt$.
Since $X$ is dense-in-itself, we can find $x_0, x_1,x_2\in X$ with
$y<x_0<x_1<x_2<y'$. Then $(y,x_1)\cap X\ne \empt\ne (x_1,y')\cap X$.
So there is $n$ such that $(y,x_1)\cap N_n\ne \empt\ne (x_1,y')\cap N_n$.
Pick $x_0'\in (y,x_1)\cap N_n$ and $x_2'\in (x_1,y')\cap N_n$.
Since $X\cap (x_0', x_2')\ne \empt$ and $N_n$ is nowhere dense,
we can find  $x_1'\in X\cap (x_0', x_2')\setm \overline{N_1}$.
Pick $U\in \mc U_n$ with $x_1'\in U$. Then $x_0', x_2'\notin U$, so 
$U\subs (x_0', x_2')\subs (y,y')$. Thus $\empt\ne U\cap X\subs X\cap (y,y')$.

\medskip

Now let $X$ be arbitrary.
Then, by Rudin's theorem, 
$X$ is the continuous image of a  compact, ordered space $Y$, $
f:Y\twoheadrightarrow X$. Then there is a closed subspace $Z$ of $Y$
such that that map $g=f\restriction Z$ is irreducible. 
Then $Z$ is a GO space, and it does not have isolated points because
$g$ is irreducible, and $X$ is dense-in-itself.

So $Z$ has a $\sigma$-disjoint $\pi$-base $\mc U$.

We claim that 
\begin{equation}
\mc V=\{X\setm {g[Z\setm U]}: U\in \mc U\} 
\end{equation}
is a $\sigma$-disjoint $\pi$-base of $X$.

First observe that if $U\in \mc U $, then $X\setm {g[Z\setm U]}\ne \empt$, i.e. 
$X\ne {g[Z\setm U]}$, because $g$ is irreducible.  Thus $\empt\notin \mc V$.

To check that $\mc V$ is a $\pi$-base, pick an arbitrary non-empty set $V\subs X$. Then  then there is $U\in \mc U$
with $U\subs g^{-1}V$. Then $X\setm g[X\setm U]\subs V$.

Finally we show that $\mc V$ is $\sigma$-disjoint.
Since $\mc U$ was $\sigma$-disjoint, it is enough to show that $U\cap U'=\empt$
implies $(X\setm f[Z\setm U])\cap  (Z\setm f[X\setm U'])=\empt$.
Indeed, assume that $(X\setm f[Z\setm U])\cap  (Z\setm f[X\setm U'])\ne \empt$.
Pick $x\in X\setm (f[Z\setm U])\cup  f[Z\setm U'])$. Fix $z\in Z$
with $g(z)=x$. Then $z\in U\cap U'$, i.e. $U\cap U'\ne \empt$.
\end{proof}

Finally, we turn our attention to a particularly interesting space: $$\sigma(2^\oo)=\{x\in 2^\oo: |x^{-1}(1)|<\omega\}$$ and to the question whether this space is $D$-separable. First, note that $\sigma(2^\oo)$ is $\sigma$-discrete and hence $d$-separable. A natural approach would then to prove that $\sigma(2^\oo)$ is DDG. However, we will show that it is independent of ZFC whether $\sigma(2^\oo)$ is DDG. 
More precisely, we prove that

\begin{theorem}\label{lm:ma_ea2}
If MA${}_{\aleph_1}$  holds then   
$\sigma(2^\oo)$ is DDG. 
\end{theorem}
\noindent and 
\begin{theorem}\label{tm:diamond_ddg}
If $\diamondsuit$  holds then   $\sigma(2^\oo)$ is not DDG.
\end{theorem}

\newcommand{\uu}[2]{U(#1,#2)}

If $a\in \br \oo;<\omega;$, then we denote the characteristic function on $a$ by $\chi_a$.
The map $a\mapsto \chi_a$  is a bijection between $\br \oo;<\omega;$ and $\sigma(2^\oo)$.
For $A\subs  \br \oo;<\omega;$
write
$\chi[A]=\{\chi_a:a\in A\}$
.

Let $$\tau=\{A\subs  \br \oo;<\omega;:\chi[A]\text{ is open in }\sigma(2^\oo)\}.$$
Instead of $\sigma(2^\oo)$ we will consider a homeomorphic copy of that space: the space  $\mc X=\<\br \oo;<\omega;, \tau\>$.

For  $x,y\in \br \oo;<\omega;$ with $x\cap y=\empt$, let
\begin{displaymath}
 \uu xy=\big\{z\in \br \oo;<\omega;: x\subs z \ \land \ y\cap a=\empt\big\}.
\end{displaymath}
If $a\in \br \oo;<\omega;$, then the family
\begin{equation}\notag
\big\{U(a,b):b\in \br \oo\setm a;<\omega;\big \}  
\end{equation}
is a neighborhood base of $a$ in $\mc X$.

\begin{lemma}\label{lm:accu}
Let $a\in \br \oo\setm a;<\omega;$ and
$B\subs \br \oo\setm a;<\omega;$.
Then $a\in B'$ iff $B$ contains an infinite $\Delta$-system with kernel $a$. 
\end{lemma}

\begin{proof}
Assume first that $a\in B'$. Choose $b_0, b_1,\dots$ from $B\setm \{a\}$ such that
\begin{equation}\notag
b_n\in B\cap \uu a{(b_0\cup\dots \cup  b_{n-1})\setm a}.
\end{equation}
Since $a\in B'$ we can construct such a sequence,
and observe that $\{b_0, b_1m\dots\}$ is an infinite $\Delta$-system with kernel $a$.

Assume now that $B$ is an infinite $\Delta$-system with kernel $a$.
If $\uu ac$ is a neighborhood of $a$, then we can pick $b\in B\setm \{a\}$ with $c\cap b=\empt$, and then
$b\in \uu ac$. So $a\in B'$.
\end{proof}

To prove Theorem \ref{lm:ma_ea2} we need the following lemma.
\begin{lemma}\label{lm:ma_ea1}
(MA${}_{\aleph_1}$)   If $E\subs \sigma(2^\oo)$ is discrete, 
$A\subs \sigma(2^\oo)$, and $E\subs A'$, then there is a discrete $D_1\subs A$ with $E_1\subs D_1'$.
\end{lemma}

\begin{proof}
For $e\in E$ pick $z(e)\in \br \oo\setm a;<\omega;$ with $E\cap \uu e{e(Z)}=\{e\}$.  
Since $e\in A'$, there is an infinite $\Delta$-system $A_e\subs \uu e{e(Z)}\cap A$ with kernel $e$.

Define $\mc P=\<P,\le\>$ as follows. Let 
\begin{multline}
P=\{p\ |\ p\in \br A\times E;<{\omega};\ \land\
\forall \<a,e\>\in p\quad \ a\in A_e\ \land\\ 
\forall \<a,e\>\ne \<a',e'\>\in p \quad a\notin \uu {a'}{z(e')}.
\}. 
\end{multline}
Let $p\le q$ iff $p\supseteq q$.

\begin{claim}\label{cl:sigma_ccc}
$\mc P$ satisfies c.c.c. 
\end{claim}

\begin{proof}[Proof of the Claim]
Assume $\{p_{\alpha}:{\alpha}<\oo\}\subs P$. 

Let $A_{\alpha}=supp (p_{\alpha})=\bigcup_{\<a,e\>\in p_{\alpha}}(a\cup z(e))$. 
There are ${\alpha}<{\beta}<\oo$ and a bijection $\sigma:A_{\alpha}\to  A_{\beta}$
such that $\sigma\restriction A_{\alpha}\cap A_{\beta}=id$, and 
\begin{equation}
 p_{\beta}=\{\<\sigma[a],\sigma[e]\>:\<a,e\>\in p_{\alpha}\}.
\end{equation}
Then $p_{\alpha}\cup p_{\beta}\in P$. Indeed, if $\<a,e\>\in p_{\alpha}\setm p_{\beta}$ and
 $\<a',e'\>\in p_{\beta}\setm p_{\alpha}$, then $a'\not\subs A_{\alpha}\cap A_{\beta}$,
so $a\notin \uu {a'}\empt$.
\end{proof}

\begin{claim}\label{cl:sigma_density}
If $p\in P$, $e\in E$, and $e\in \uu ez$,  then there is $a\in A\cap \uu ez$
such that $p\cup\{\<a,e\>\}\in P$.  
\end{claim}

\begin{proof}[Proof of the Claim]
 
There is $a\in A_e$ such that 
\begin{enumerate}[(D1)]
 \item $e\subs a$,
\item\label{D2} $(supp(p)\cup z)\cap a=e$,
\item $a\setm supp(p)\ne \empt$.
\end{enumerate}
We claim that this $a$ works.

Indeed, if $\<a',e'\>\in p$, then $a'\not\supset a$, and so $a'\notin \uu a\empt\supset \uu a{z(e)}$.
On the other hand, assume on the contrary that $a\in  \uu {a'}{z(e')}$. Then
$e\in \uu {a'}{z(e')}$ by (D\ref{D2}). Thus $e=e'$. Thus $a,a'\in A_e$, and so
$a\notin \uu {a'}\empt$. Contradiction. 
\end{proof}

For $e\in E$ and $n\in \omega$ let
\begin{equation}\notag
 \mc D_{e,n}=\{p\in P: |\{a: \<a,e\>\in p\}|\ge n\}.
\end{equation}
By Claim \ref{cl:sigma_density}, the sets $\mc D_{e,n}$ are dense.
Let 
\begin{equation}\notag
\mbb D=\{D_{e,n}:e\in E, n\in \omega \}. 
\end{equation}
By Claim \ref{cl:sigma_ccc}, $MA_{\aleph_1}$ implies that 
there is a $\mbb D$-generic filter $\mc G\subs P$.
For $e\in E$ let 
\begin{equation}\notag
 F_a=\{a\in A: \<a,e\>\in \bigcup \mc G\},
\end{equation}
and 
\begin{equation}\notag
D_1=\bigcup\{F_e:e\in E\}. 
\end{equation}
Then $D_1$ is discrete, because if $\<a,e\>\in \bigcup \mc G$, then 
$D_1\cap \uu a{z(e)}=\{a\}$ by the construction of the poset.
Moreover,  for $e\in E$ the set   $F_e$ is infinite by the genericity of $\mc G$, and 
$F_e$ is a $\Delta$-system with kernel $e$ by the construction of the poset. So $e\in F'_e\subs D_1'$

Thus we proved the lemma.
\end{proof}

We now finish with the proof of the first theorem:

\begin{proof}[Proof of Theorem \ref{lm:ma_ea2}]
Let  $A\subs \sigma(2^\oo)$, and let $E\subs \overline{A}$ be a discrete set. 
We need to find  a discrete set $D\subs A$ with $E\subs \overline D$.

To start with   let $E_0=\{e\in E\cap A: e\notin A'\}$.

Then $E_0$ is discrete, moreover $E_0\cap \overline{A\setm E_0}=\empt$.
Let $A_1=A\setm \overline{E_0}$ and $E_1=E\setm \overline{E_0}$.
Then $E_1\in A_1'$, so, by Lemma \ref{lm:ma_ea1}, there is a discrete set  $D_1\subs A_1 $
with $E_1\subs D_1'$.  
Then 
\begin{enumerate}[(I)]
 \item $E_0\cap \overline{D_1}=\empt$ because $E_0\cap A'=\empt$, and 
 \item $\overline{E_0}\cap D_1=\empt$ because $D_1\subs A_1=A\setm \overline{E_0}$.
\end{enumerate}
Thus $D=E_0\cup D_1$ is the required set.
\end{proof}

%

\begin{proof}[Proof of Theorem \ref{tm:diamond_ddg}]
%
%
Consider the discrete subspace  $D=\br \oo;1;$ of $\sigma(2^\oo)$.
Using $\diamondsuit$ we will construct $\mc A\subs \br \oo;\ge 2;$
such that $D\subs \mc A'$, but there is no discrete $\mc E\subs \mc A$
with $D\subs \mc E'$. 

Fix a $\diamondsuit$-sequence $\<\mc A_{\alpha}:{\alpha}<\oo\>$
which guesses subsets of $\br \oo;<{\omega};$, i.e.
\begin{equation}
\forall \mc C\subs \br \oo;<{\omega};\
\{{\alpha}:  \mc C\cap \br {\alpha};<{\omega};=\mc C_{\alpha}\}
\text{ is stationary}. 
\end{equation}

We construct a continuous sequence 
$\<\mc A_{\beta}:{\omega}\le {\beta}\le \oo\>$ such that 
\begin{enumerate}[(a)]
 \item $\mc A_{\beta}\subs \{b\in \br {\beta};<{\omega};:|b|\ge 2\}$,
\item 
$\bigl\{\{{\alpha}\}:{\alpha}<{\beta}\bigr\}\subs \mc A'_{\beta}$
\item if  $b\in \mc A_{\beta}$ and 
$|b\cap {\nu}|\ge 2$ for some ${\nu}<{\beta}$, 
then $b\cap {\nu}\in \mc A_{\nu}$,
\end{enumerate}
%
%
%
as follows.
Let 
\begin{equation}
 \mc A_{{\omega}}=\{x\in \br {\omega};<{\omega};: |x|\ge 2\}.
\end{equation}
If ${\beta}$ is limit, let 
\begin{equation}
\mc A_{\beta}=\bigcup\{\mc A_{\zeta}:{\alpha}<{\beta}\}. 
\end{equation}
Assume that ${\beta}={\alpha}+1$.

If 
$\mc C_{\alpha}\not\subs \mc A_{\alpha}$
or there is ${\nu}<{\alpha}$ such that $\{{\nu}\}$ is not an accumulation point 
of $\mc C_{\alpha}$ then let  
\begin{equation}
\mc A_{{\alpha}+1}=\mc A_{\alpha}\cup\bigl\{\{{\alpha},2n,2n+1 \}:n<{\omega}
\bigr\}.
\end{equation}
Assume now that $\mc C_{\alpha}\subs \mc A_{\alpha}$ and
\begin{equation}
\bigl\{\{{\nu}\}:{\nu}<{\alpha}\bigr\}\subs \mc C'_{\alpha}. 
\end{equation}
Then for each ${\nu}<{\alpha}$ there is a $\Delta$-system
in $\mc C_{\alpha}$ with kernel ${\nu}$. So there are  
infinitely many pairwise disjoint elements 
$\{b^{\alpha}_n:n<{\omega}\}$ in $\mc C_{\alpha}$.
Let
\begin{equation}
\mc A_{{\alpha}+1}=\mc A_{\alpha}\cup\bigl\{\{{\alpha}\}\cup b^{\alpha}_n:n<{\omega}
\bigr\}.
\end{equation}

It is clear that (a)-(c) hold.

Let $\mc A=\mc A_{\oo}$.
Then $\bigl\{\{{\alpha}\}:{\alpha}<{\oo}\bigr\}\subs \mc A'$ by (b).
Assume on the contrary that there is a discrete set 
$\mc E\subs \mc A$ with $\bigl\{\{{\alpha}\}:{\alpha}<{\oo}\bigr\}\subs \mc E'$.
For each $e\in E$ fix neighborhood $\uu e{z(e)}$  of $e$ with
$\uu e{z(e)}\cap \mc E=\{e\}$.

Then there is ${\alpha}<\oo$ such that 
\begin{multline}\notag
\bigl\{\{{\nu}\}:{\nu}<{\alpha}\bigr\}\subs (\mc E\cap \br {\alpha};<{\omega};)',
z(e)\subs {\alpha} \text{ for $e\in \mc E\cap \br {\alpha};<{\omega};$},\\ 
\text{ and }
 \mc C_{\alpha}=\mc E\cap \br {\alpha};<{\omega};. 
\end{multline}
Pick $b\in \mc E\cap \uu {\{\alpha\}}\empt$.
Since $\{\{{\nu}\}:{\nu}<{\alpha}\bigr\}\subs \mc C_{\alpha}'$, 
we have  $b\cap ({\alpha}+1)=\{{\alpha}\}\cup b^{\alpha}_n$ for some $n<{\omega}$.
Since $b^{\alpha}_n\in \mc C_{\alpha}\subs \mc E$, we also have
\begin{equation}\notag
 \uu {b^{\alpha}_n}{z(b^{\alpha}_n)}\cap \mc E=\{b^{\alpha}_n\}.
\end{equation}
But $b^{\alpha}_n\cup z(b^{\alpha}_n)\subs {\alpha}$ and $b\cap {\alpha}=b^n_{\alpha}$,
so 
\begin{equation}
b\in  \uu {b^{\alpha}_n}{z(b^{\alpha}_n)}.
\end{equation}
Contradiction.
\end{proof}

The following remains unsolved:

\begin{question}
 Is it provable in ZFC that $\sigma(2^\oo)$ is $D$-separable? Does $\Diamond$ imply that $\sigma(2^\oo)$ is not $D$-separable?
\end{question}

\section{Acknowledgements}

The authors acknowledge helpful comments from Justin T. Moore and the the Set Theory and Topology groups of both the R\'enyi and Fields Institute.

\end{document}